\documentclass[11pt]{article}
\usepackage[mathscr]{euscript}
\usepackage{amsfonts,amsmath,amssymb,amsthm}
\usepackage{graphicx}
\newcommand{\image}[2] %Image scaled centered
{     \begin{center} %Image
	\includegraphics[scale=#1]{#2.png}
      \end{center}}
\usepackage{aliascnt}
\usepackage{hyperref}
\hypersetup{colorlinks=false, pdfborder={0 0 0}}
\newcommand{\nref}[1]{\hyperref[#1]{\ref*{#1}}}

\newcommand{\U}{\mathcal{U}} 	% Universal Enveloping Algebra
\renewcommand{\O}{\mathcal{O}} 	% Matrix Function Algebra
\newcommand{\B}{\mathfrak{B}}   % Nichols algebra
\newcommand{\T}{\mathfrak{T}} 	% Tensor algebra
\newcommand{\g}{\mathfrak{g}} 	% arbitrary Lie algebra
\newcommand{\Z}{\mathbb{Z}}  	% Integers
\newcommand{\C}{\mathbb{C}}  	% Complex
\renewcommand{\k}{\Bbbk}  	% arbitrary Field
\newcommand{\K}{\Bbbk}          % dito
\newcommand{\F}{\mathbb{F}}  	% finite field
\newcommand{\D}{\mathbb{D}}  	% Dihedral
\renewcommand{\S}{\mathbb{S}}  	% Symmetric
\newcommand{\A}{\mathbb{A}}  	% Alternating
\newcommand{\N}{\mathbb{N}}     % Naturals
\renewcommand{\H}{\mathcal{H}} %\Griess algebra

\newcommand{\BB}{\mathbb{B}}

\newcommand{\ydm}{Yetter-Drinfel'd module }
\newcommand{\ydms}{Yetter-Drinfel'd modules }
\newcommand{\ydmP}{Yetter-Drinfel'd module}

\theoremstyle{plain}
\newtheorem{theorem}{Theorem}[section]

\newtheorem{corollary}[theorem]{Corollary}

\newtheorem{definition}[theorem]{Definition}

\newtheorem{example}[theorem]{Example}

\newtheorem{lemma}[theorem]{Lemma}

\newtheorem{remark}[theorem]{Remark}

\DeclareMathOperator{\tr}{tr}

\DeclareMathOperator{\Aut}{Aut}

\DeclareMathOperator{\Env}{Env}
\newcommand{\id}{\mathrm{id}}

\DeclareMathOperator{\End}{End}
\DeclareMathOperator{\SL}{SL}
\DeclareMathOperator{\cha}{char}
\DeclareMathOperator{\Lin}{Lin}
\DeclareMathOperator{\Cent}{Cent}

\newcommand{\trace}[2]{\tr^{#1}_{#2}}
\newcommand{\traceBVt}[1]{\tr^{#1}_{\B(M)}(t)}
\newcommand{\trid}{\triangleright}

\newcommand{\Perm}[1]{\S_{#1}}
\newcommand{\APerm}[1]{\A_{#1}}
\newcommand{\op}{\mathrm{op}}
\newcommand{\ot}{\otimes}
\newcommand{\sort}{\mathrm{sort}}

\begin{document}

\title{Factorization of Graded Traces on Nichols Algebras}

\author{
    Simon Lentner \\
    Algebra and Number Theory (AZ),\\
    Universit\"at Hamburg \\
    Bundesstra{\ss}e 55, D-20146 Hamburg \\
    simon.lentner(at)uni-hamburg.de \\
    \and
    Andreas Lochmann\footnote{corresponding author} \\
    Fachbereich Mathematik und Informatik, \\
    Philipps-Universit\"at Marburg \\
    Hans-Meerwein-Stra{\ss}e, D-35032 Marburg \\
    lochmann(at)mathematik.uni-marburg.de \\
}

\maketitle

\begin{abstract}
We study the factorization of the Hilbert series and more
general graded traces of a Nichols algebra into cyclotomic polynomials. Nichols
algebras play the role of Borel subalgebras of finite-dimensional quantum
groups, so this observation can be viewed as an analog to the factorization of
the order of finite Lie groups into cyclotomic polynomials. We prove results on
this factorization and give a table of many examples of rank 1 over nonabelian
groups.
\end{abstract}

% TODO: Keywords, Classification

\tableofcontents

\section{Introduction}

A Nichols algebra $\B(M)$ over a finite group is a graded algebra and in
particular a braided Hopf algebra, that plays the role of a Borel
subalgebra of a finite-dimensional quantum group. For example, the Borel
subalgebras $u_q^\pm(\g)$ of the Frobenius-Lusztig kernels $u_q(\g)$ are
finite-dimensional Nichols algebras over quotients of the root lattice of $\g$
as an abelian group.

In this article we discuss the following curious phenomenon, that is apparent
throughout the ongoing classification of finite-dimensional Nichols algebra:
The \emph{graded dimension} or \emph{Hilbert series} of any finite-dimensional
Nichols algebra known so far factorizes as a polynomial in one variable into the
product of cyclotomic polynomials. For Nichols algebras over abelian groups,
this factorization can be explained (see below), but this mechanism cannot
explain the empirically observed complete factorization of the Hilbert series of
Nichols algebras over nonabelian groups.

 However, it has been shown in a series of
joint papers of the second author (\cite{GHV}, \cite{HLV}, \cite{HLV2}) that the
existence of such a complete factorization into cyclotomic polynomialy implies
strong bounds on the number of relations in low degree in the Nichols algebra. On
the other hand, the first author has obtained in \cite{Len12} new families of
Nichols algebras over nonabelian groups, where such a complete factorization
appears by construction.

A similar factorization into cyclotomic
polynomials appears for the order of any finite simple group of Lie type, which
is due to the action of the Weyl group on the space of polynomial functions on
the root lattice (\cite{Carter85}, Sec. 2.9). Moreover, the existence of such a
factorization continues to characters of unipotent representations
(\cite{Carter85}, Sec 13.7--13.9). This observation may be less surprising, due to
the fact that Lusztig has related the representation theory of $u_q(\g)$ over
$\C$ to the representation theory of the corresponding Lie group over a finite
field. The precise connection seems however unclear at this point. At least we
shall see below that the Weyl-group theory of Nichols algebras indeed generally
implies a factorization of the Hilbert series and also of graded characters, but
for nonabelian groups, this is not the finest factorization we observe.

More concretely, a key observation in this article
is that in the available examples not only the Hilbert series itself (the
graded trace of the identity), but arbitrary \emph{graded characters} of the
group acting on the Nichols algebras factorize into cyclotomic polynomials. 
The aim of the present article is to gather and prove basic facts about this
factorization, calculate a comprehensive list of examples, and start with
providing systematic factorization mechanisms.

\bigskip

\noindent
The content of this article is as follows:

In Section 2 we define graded traces and prove the relevant basic facts
about graded traces on finite-dimensional Nichols algebras, including
additivity and multiplicativity with respect to the representations,
rationality, and especially Poincar\'e duality. Most of these
contents have appeared in literature in some form, we gather them here for
convenience. We also give an example for a graded representation, where the
graded characters factor in
a nice way which is not reflected by the represention itself. This
phenomenon will occur when we study the characters of Nichols algebras
over non-abelian groups.

In Section 3 we study factorization mechanisms for Nichols algebras, and
hence for their graded traces and especially their Hilbert series.
The root system $\Delta^+$ of a Nichols algebra $\B(M)$ in the sense of
\cite{AHS10} directly presents a factorization of $\B(M)$ as graded vectorspace:
This completely explains the factorization of the graded trace of an
endomorphism $Q$ that respects the root system grading.

For Nichols algebras over abelian groups, we use the theory of Lyndon
words to significantly weaken the assumption on $Q$. On the other hand, we
give an example of an endomorphism (the outer automorphism of $A_2$ containing
a loop) where this mechanism fails; a factorization of the graded trace is
nevertheless observed and can be tracked to the surprising existence of an
 alternative ``symmetrized'' PBW-basis. The formulae appearing involve the
orbits of the roots under $Q$ and are in resemblance to the formulae given in
\cite{Carter85} Sec. 13.7 for finite Lie groups.

For Nichols algebras over nonabelian groups, the root system
factorization discussed above still applies, but is too crude in general to
explain the full observed factorization into cyclotomic polynomials. We give
however a family of examples constructed as covering Nichols algebras by the
first author in \cite{Len13a}, where the complete factorization can be tracked
back to a finer PBW-basis. The key ingredient is an additional root system in a
symplectic vector space over $\F_2$ (see \cite{Len13b}), a structure appearing
as well in the representation theory of finite Lie groups, see \cite{Carter85}
Sec. 13.8.

Finally we give results on the divisibility of the Hilbert series derived by
the second author in \cite{Lochmann_divisibility}. By the freeness of the
Nichols algebra over a sub-Nichols algebra (\cite{Skr_freeness}, \cite{G_freeness})
one can derive a divisibility of the graded trace $\trace{g}{}$ by that of a
sub-Nichols algebra. Moreover, in many cases there is a shift-operator $\xi_x$
for some $x\in G$, which can be used to prove that there is an additional
cyclotomic divisor of the graded trace. This holds in particular for $x\in G$
commuting with $g$.

Section 4 finally displays a table of graded characters for all known examples 
of finite-dimensional Nichols algebras of non-abelian group type and rank $1$,
which were computationally accessible to us. With respect to the previous
discussion, the rank $1$ case is particularly challenging, as the Weyl group is
trivial. We observe that again all graded characters factorize in these examples.

             % Section 1: Introduction
\section{Hilbert Series and Graded Traces}
\subsection{The Graded Trace \texorpdfstring{$\trace{Q}{V}$}{trQ(V)}}
In the following we suppose $V=\bigoplus_{n\geq 0} V_n$ to be a graded
vector space with finite-dimensional layers $V_n$. We frequently call a linear
map $Q:V\rightarrow V$ an \emph{operator} $Q$. We denote the identity operator
by $1_V=\oplus_{n\geq 0}1_{V_n}$ and the projectors to each $V_n$ by $P_n$. An
operator $Q$ is called \emph{graded} if one of the following equivalent
conditions is fulfilled:
\begin{itemize}
 \item $Q$ is commuting with all projections $P_n$.
 \item $Q$ preserves all layers $V_n$. 
\end{itemize}
We denote the restriction of a graded operator $Q:V\rightarrow V$ to each $V_n$
by $Q_n\in \End(V_n)$. An operator $Q$ is called \emph{algebra operator}
resp.\ \emph{Hopf algebra operator} if $V$ is a graded algebra resp.\ Hopf algebra and
$Q$ is an algebra resp.\ Hopf algebra morphism.

\begin{definition}
For a graded vector space $V$ with finite-dimensional layers define 
\begin{itemize}
  \item the \emph{Hamilton operator} $E\in End(V)$ by $E|_{V_n}=n$, and
  \item the \emph{Boltzmann operator} $t^E\in End(V)[[t]]$ as $End(V)$-valued
    series $t^E=\sum_{n\geq 0}t^n\cdot P_n$, i.e.\ $t^E|_{V_n}=t^n$.
\end{itemize}
\end{definition}
\begin{remark}
  Let $V$ be a graded algebra, then the Boltzmann operator is an automorphism 
  of the graded algebra $V[[t]]$.
\end{remark}
Note that the trace of the Boltzmann operator is the \emph{Hilbert series}
$\tr\left(t^E\right)=\sum_{n\geq 0} \dim(V_n)t^n=:\H(t)$.
If $V$ is finite-dimensional, $\H(t)$ is a polynomial and has a well-defined
value at $t=1$, which turns out to be the dimension:
\begin{align*}
    \tr\left(t^E\right)|_{t=1}
    &=\tr\left(1^E\right)
    =\tr\left(1_V\right)=\dim\left(V\right)
  \end{align*}

\begin{definition}
For $V$ a graded vector space with finite-dimensional layers and $Q$ a graded
operator, define the \emph{graded trace} of $Q$ as the series
$\trace{Q}{V}(t):=\tr\left(t^EQ\right)
  =\sum_{n\geq 0} t^n\cdot \tr\left(Q_n\right)$
in the variable $t$.
\end{definition}

\begin{example}
In particular the identity $Q=1_V$ has as its graded trace 
$\trace{1}{V}(t)$ again the Hilbert
series $\tr\left(t^E\right)=\H(t)$.
More generally, for $\lambda\in\k^\times$, the scaling
operator $Q=\lambda 1_V$ has graded trace
$\trace{\lambda}{V}(t)
=\tr\left( t^E\cdot \lambda1_V\right)
=\lambda\cdot\tr\left( t^E\right)
=\lambda\cdot\trace{1_V}{V}(t)=\lambda\cdot\H(t)$.
\end{example}

\begin{remark}
  Assume $V$ is finite-dimensional. Then only
  finitely many layers $V_n$ are non-trivial. Thus $\trace{Q}{V}(t)$ is a
  polynomial in the formal variable $t$ with $\trace{Q}{V}(1)=\tr(Q)$.
  If $A$ is an infinite-dimensional algebra, it is proven in
  \cite{Jing97} that $\trace{Q}{A}$ is a rational function if either
  \begin{itemize}
   \item $A$ is commutative and finitely generated,
    \item $A$ is right noetherian with finite global dimension, or
    \item $A$ is regular.
  \end{itemize}
\end{remark}

\begin{lemma}
  The graded trace $\trace{Q}{V}(t)$ is linear in $Q$.
%, namely for $Q=\lambda
%  R+\eta S$ with $\lambda,\eta\in \k^\times$ we have
%  \begin{align*}
%    \trace{\lambda R+\eta S}{V}(t)
%     &=\tr\left(t^E\left(\lambda R+\eta S\right)\right) \\
%    &=\lambda\cdot\tr\left(t^ER\right)+\eta\cdot\tr\left(t^ES\right)\\
%    &=\lambda\cdot \trace{R}{V}(t)+\eta\cdot\trace{S}{V}(t)
%  \end{align*}
\end{lemma}

\begin{lemma}\label{lm_product}
  For graded vector spaces $V,W$ with finite-dimensional layers, the sum
  $V\oplus W$ and product $V\otimes W$ again are graded
  vector spaces with finite-dimensional layers. The respective codiagonal grading
  by definition implies
  $$t^{E_{V\oplus W}}=t^{E_V}\oplus t^{E_W}
  \qquad \text{and} \qquad t^{E_{V\otimes W}}=t^{E_V}\otimes t^{E_W}\,. $$
  Then the following properties for the graded trace hold
  immediately from the respective properties of the trace:
  $$\trace{Q\oplus R}{V\oplus W}(t)
    =\trace{Q}{V}(t)+\trace{R}{W}(t)
  \qquad \text{and} \qquad
    \trace{Q\otimes R}{V\otimes W}(t)
    =\trace{Q}{V}(t)\cdot\trace{R}{W}(t)\,. $$
%   Especially on the graded vectorspaces $V=\O=\{0\}$ and $V=\I$
%   (with only $\I_0=e\k$ and other $\I_k=\{0\}$) the Hamilton operator $t^E$
%   acts trivially and the respective graded traces of any $Q$ is:
%   \begin{align*}
%     \trace{Q}{\O}(t)
%     &=0 \\
%     \trace{Q}{\I}(t)
%     &=\tr\left( Q_0 \right) 
%   \end{align*}
\end{lemma}

\begin{remark}
  The Boltzmann operator $t^E$ is used in a thermodynamic
  ensemble to stochastically average over eigenspaces of the Hamiltonian $E$
  weighted by the eigenvalue of $E$ (``energy'') and depending on a
  free  parameter $t$ (for $t=e^{-\beta}$ with $\beta$ the ``inverse
  temperature''), see \cite{Sak11} Chp. 3, especially 3.4.50--3.4.53.\\

  Equivalenty, after a so-called Wick rotation, the Boltzmann operator is
  the time-development operator in a quantum field theory, see \cite{Sak11}
  2.6.21: 
  $$U(T_1,T_2)=U(T_2-T_1)=t^E,\qquad t=e^{-i\Delta T/\hbar}$$
  
  The trace of $t^E$ is called \emph{partition function} $Z$ and yields the
  Hilbert series is in our situation.
  To yield a probablity distribution one normalizes the Boltzmann operator
  $t^E$ by the partition function to obtain trace $1$ and the result is
  known as the quantum \emph{density operator}
  $\rho:=t^E/\tr\left(t^E\right)$.
  Finally, then up to the factor $Z$, the graded trace is
  the quantum \emph{expectation value}
  \begin{align*}
  \langle Q\rangle
  &:=\tr\left(\rho\, Q\right)
  =\frac{\tr\left(t^EQ\right)}{\tr\left(t^E\right)}
  =\frac{\trace{Q}{V}(t)}{\H(t)} 
  \end{align*}
\end{remark}

              % Section 2: Hilbert Series and Graded Traces
  \subsection{Poincar\'e Duality}

The Hilbert series of a Nichols algebra exhibits a Poincar\'e Duality,
see e.g. \cite{MS00} Rem. 2.2.4. We generalize this approach to the calculation
of arbitrary algebra operator traces in Nichols algebras:

\begin{lemma} \label{LEM_trace_of_inverse}
Let $\B(M)$ be a finite-dimensional Nichols algebra with top degree $L$
and integral $\Lambda$. Let $Y$ be an arbitrary algebra automorphism of
$\B(M)$ and the scalar $\lambda_Y\in\k^\times$ such that
$Y\Lambda=\lambda_Y\cdot \Lambda$. Then 
$\tr\left(Y\right)=\lambda_Y\cdot \tr\left(Y^{-1}\right)$.
\end{lemma}
\begin{proof}
Let $\{b_j\}_{j\in J}$ and $\{b^*_j\}_{j\in J}$ be two bases of $\B(M)$ with
$b^*_ib_j=\delta_{ij}\cdot \Lambda$ for all $i,j\in J$.
%This implies
%for any graded algebra automorphism $Y$ acting on $\Lambda$ by the scalar $\lambda_Y$:
Then holds
\begin{align*}
\tr\left(Y\right)\cdot\Lambda
\;&=\; \sum_{j\in J} \left(Yb^*_j\right)b_j
\;=\; Y\sum_{j\in J} b^*_j\left(Y^{-1}b_j\right)
\;=\; Y\left(\tr\left(Y^{-1}\right)\cdot\Lambda\right)\\
&=\; \tr\left(Y^{-1}\right)\cdot Y\Lambda
\;=\; \lambda_Y\cdot \tr\left(Y^{-1}\right)\cdot \Lambda\,.
\end{align*}
\end{proof}

\begin{corollary} \label{COR_poincare_duality}
Let $\B(M)$ be a finite-dimensional Nichols algebra with top degree $L$
and integral $\Lambda$. Let $Q$ be an arbitrary algebra automorphism of
$\B(M)$ and the scalar $\lambda_Q\in\k^\times$ such that
$Q\Lambda=\lambda_Q\cdot \Lambda$. Then 
$$\trace{Q}{\B(M)}(t)=\lambda_Q\cdot t^L\cdot \trace{Q^{-1}}{\B(M)}(t^{-1}) $$
\end{corollary}
\begin{proof}
We apply Lemma \ref{LEM_trace_of_inverse} to $Y=t^EQ$ and $\lambda_Y=t^L\lambda_Q$:
\begin{align*}
\trace{Q}{\B(M)}(t)
\;&=\; \tr\left(t^EQ\right)
\;=\; t^L\lambda_Q\cdot \tr\left(\left(t^EQ\right)^{-1}\right)
\;=\; t^L\lambda_Q\cdot \tr\left(Q^{-1}\left(t^{-1}\right)^E\right) \\
\;&=\; t^L\lambda_Q\cdot \tr\left(\left(t^{-1}\right)^EQ^{-1}\right)
\;=\; \lambda_Q\cdot t^L \cdot \trace{Q^{-1}}{\B(M)}(t^{-1})
\end{align*}
\end{proof}

The special case $Q=Q^{-1}=1_M$ recovers the Poincar\'e duality 
$\H(t)=t^L\cdot\H(t^{-1})$ of the Hilbert series,
and therefore $\dim\B(M)_l=\dim\B(M)_{L-l}$ for all $l$.

  \subsection{An Example for Factorization only in the Trace} \label{SEC_toy_example}

We want to present a type of graded representations which exhibits
a seemingly paradoxical property: Their graded characters factor
nicely, whereas the representations themselves do not. We will
encounter this property in the case of Nichols algebras of
non-abelian group type of rank 1.

%\begin{Proposition}
%Let $H\hookrightarrow G\twoheadedrightarrow A$ be an exact sequence
%of finite groups, and assume that $A$ is abelian. $G$ acts on $\K G$
%by conjugation. $H$ is normal, and hence $\K H$ is a subrepresentation
%of $\K G$. Let $f_H: H\rightarrow \Z$ be a $G$-invariant function and
%$f_A: A\rightarrow \Z$ arbitrary. Then
%$f:G\rightarrow \Z$, $f(hr):=f_H(h)+f_A\circ\pi(r)$ defines
%a $G$-invariant function on $G$.
%\end{Proposition}

\begin{theorem} \label{THE_toy_example}
Let $G$ be a finite group and let $\pi:G\twoheadrightarrow A$ be some
epimorphism into an abelian group $A$.
Let $A=\bigoplus_{j\in J}\,A_j$ be some decomposition of $A$ into
cyclic groups $A_j\cong \Z/n_j\Z$. Let $f_j:A_j\rightarrow\Z$
be the set-theoretic sections with $f_j(0)=0$ and $f_j(x+1)=f_j(x)+1$ for all
$x\in A_j\setminus\{-1\}$. Set $f:=\sum_{j\in J}\,f_j:A\rightarrow \Z$.
Then $\pi$ defines $A$- and $A_j$-gradings of $\K G$
as algebra and $f\circ\pi$ defines a $\Z$-grading of $\K G$ as
vector space. Let $G$ act on $\K G$ by conjugation. This action respects
the grading and $\K G$ becomes a $\Z$-graded representation.
Then $\trace{g}{\K G}(t)$ is an integer multiple of $\prod_{j\in J}\,
\big(\frac{n_j}{m_j}\big)_{t^{m_j}}$ for any $g\in G$ and suitable $m_j\in\N$
(depending on $g$). (Note that this depends on the chosen
decomposition of $A$.)
\end{theorem}

To prove Theorem \ref{THE_toy_example}, we will need a
short lemma.

\begin{lemma} \label{LEM_toy_example_centralizer}
Let $\pi:G\rightarrow A$ be a homomorphism of finite groups
and $g\in G$ arbitrary. Denote the centralizer of $g$ with $Z(g)$.
Then $\#(Z(g)\cap\pi^{-1}(a))$ is either zero or $\#(Z(g)\cap\pi^{-1}(e))$,
and $A':=\{a\in A\,:\,Z(g)\cap\pi^{-1}(a)\neq\emptyset\}$
is a subgroup of $A$.
\end{lemma}
\begin{proof}
Set $Z_a := Z(g)\cap\pi^{-1}(a)$. Let $a$, $b\in A$ be arbitrary,
such that $Z_a,Z_b\neq\emptyset$. Choose $x_0\in Z_a$ and $y_0\in Z_b$.
For any $y\in Z_b$, we find that $x_0yy_0^{-1}\in Z_a$, because
$Z(g)$ is a subgroup and $\pi(x_0yy_0^{-1})=a$. Furthermore, letting
$y$ run, the $x_0yy_0^{-1}$ are $\#Z_b$ pairwise different elements,
thus $\#Z_b\leq \#Z_a$, and by symmetry, we find $\#Z_b=\#Z_a$.
Clearly, $e\in Z_e$, so all non-empty $Z_a$ have the same cardinality
$\# Z_e$.

We have $e\in Z_e$, thus $e\in A'$. Now let $a,b\in A'$ be arbitrary
and choose $x\in Z_a$, $y\in Z_b$. Then $xy\in Z_{ab}$ and $x^{-1}\in
Z_{a^{-1}}$, thus $ab, a^{-1}\in A'$,
and $A'$ must be a subgroup of $A$.
\end{proof}

\noindent
\begin{proof}{\em [Theorem \ref{THE_toy_example}]}\quad
Denote with $\iota_j:A\rightarrow A_j$ the canonical projection
and set $\pi_j:=\iota_j\circ\pi$. Let $\chi_j(g)$
be the graded character of the adjoint representation of $G$ on
$\K G$, where the grading is induced by $\pi_j$.
The adjoint representation is a permutation representation, and
thus its character is given by the number of fixed points:
$$\chi_j(g)
%\;&=\; \sum_{a\in A_j}\,\#\big\{h\in \pi_j^{-1}(a)\,:\,
%ghg^{-1}=h\big\}\,t^{f_j(a)}
\;=\; \sum_{a\in A_j}\,\#\big(Z(g)\cap \pi_j^{-1}(a)\big)\,t^{f_j(a)}$$
We denote $Z^{(j)}_a:=Z(g)\cap\pi_j^{-1}(a)$ and
$Z_a:=Z(g)\cap\pi^{-1}(a)$ for $a\in A_j$ and $a\in A$, respectively.
According to Lemma \ref{LEM_toy_example_centralizer}, there
is a subgroup $A'_j$ of $A_j$, such that
$\chi_j(g) \,=\, \#Z_e^{(j)}\cdot \sum_{a\in A'_j}\,t^{f_j(a)}$.
The subgroup of a cyclic subgroup is cyclic, so there is some
$m_j\in\N$ with $A'_j=m_jA_j$ and
$$\chi_j(g) \;=\; \#Z_e^{(j)}\cdot \sum_{b=1}^{\#A'_j}\,t^{bm_j}
\;=\; \#Z_e^{(j)}\cdot \left(\frac{n_j}{m_j}\right)_{t^{m_j}}\,.$$
Set $k:=\# J$. By definition,
\begin{eqnarray} \label{EQ_proof_toy_1}
\trace{g}{\K G}(t) \;= \sum_{a\in A}\,\#Z_a\,t^{f(a)} 
\;= \sum_{a_1\in A_1}\!\!\cdots\!\! \sum_{a_k\in A_k}\,\#Z_{a_1+\cdots+a_k}\,
t^{f_1(a_1)}\,\cdots\,t^{f_k(a_k)}\,.
\end{eqnarray}
Now $Z_{a_1+\cdots+a_k}=Z_{a_1}^{(1)}\cap\cdots\cap Z_{a_k}^{(k)}$.
By Lemma \ref{LEM_toy_example_centralizer}, we can normalize their cardinalities
to lie within $\{0,1\}$, replace $\cap$ by multiplication, and find
$$ \frac{\#Z_{a_1+\cdots+a_k}}{\#Z_e}
\;\;=\;\; \prod_{j\in J}\,
\frac{\#Z_{a_j}^{(j)}}{\#Z_e^{(j)}}\,. $$
Inserting this into Equation \ref{EQ_proof_toy_1} yields
\begin{eqnarray*}
\trace{g}{\K G}(t)
&=& \sum_{a_1\in A_1}\cdots \sum_{a_k\in A_k}\,\#Z_e\cdot \prod_{j\in J}\,
\frac{\#Z_{a_j}^{(j)}}{\#Z_e^{(j)}}\,t^{f_j(a_j)} \\
&=& \frac{\#Z_e}{\prod_{j\in J}\,\#Z_e^{(j)}}\cdot
\sum_{a_1\in A_1}\cdots \sum_{a_k\in A_k}\, \prod_{j\in J}\,
\#Z_{a_j}^{(j)}\,t^{f_j(a_j)} \\
&=& \frac{\#Z_e}{\prod_{j\in J}\,\#Z_e^{(j)}}\cdot
\prod_{j\in J}\, \sum_{a_j\in A_j}\, \#Z_{a_j}^{(j)}\,t^{f_j(a_j)} \\
&=& \frac{\#Z_e}{\prod_{j\in J}\,\#Z_e^{(j)}}\cdot
\prod_{j\in J}\, \chi_j(g)
\;\;=\;\; \#Z_e\cdot
\prod_{j\in J}\,\left(\frac{n_j}{m_j}\right)_{t^{m_j}}\,.
\end{eqnarray*}
\end{proof}

Let $\pi_1:G_1\twoheadrightarrow A_1$ be as in Theorem
\ref{THE_toy_example}
and set $G_2:=\ker \pi_1$. As a subgroup of $G_1$, $G_2$ acts on
$G_1$ by conjugation, thus $\K G_1$ is a $\Z$-graded
$G_2$-representation with $\K G_2$ in degree zero.
Assume there is an epimorphism
$\pi_2:G_2\twoheadrightarrow A_2$ as in the theorem.
Then $\K G_2$ becomes a $\Z$-graded $G_2$-representation.
A transversal $R$ of $G_2$ in $G_1$ now defines a new
$\Z$-grading on $\K G_1$ by $\deg(gr):=\deg(g)+\deg(r)$ for
all $g\in G_2$ and $r\in R$, which makes $\K G_1$ into
a $G_2$-representation with a $[G_2:A_2]$-dimensional
degree-zero-subspace. If $G_1$ is solvable, we may
use induction until the final group $G_k$ is abelian by
itself, so $\K G_1$ becomes a $G_k$-representation
with one-dimensional degree-zero-sub\-space. The grading
however still depends on the chosen subnormal series.

\begin{example}
Let $G$ be the dihedral group $G=\D_4=\langle a,b\;:\;
a^4=b^2=e,\,ab=ba^3\rangle$, acting on its group algebra
by conjugation, and let $A$ be its abelianization.
The corresponding $\Z$-graduation of $\K G$ is
%\begin{eqnarray*}
%\K G(0) &=& \Lin_\K(e, a^2) \\
%\K G(1) &=& \Lin_\K(a, b, a^3, a^2b) \\
%\K G(2) &=& \Lin_\K(ab, a^3b)
%\end{eqnarray*}
$$ \K G \;=\; \Lin_\K(e, a^2)\;\oplus\;
\Lin_\K(a, b, a^3, a^2b)\cdot t\;\oplus\;
\Lin_\K(ab, a^3b)\cdot t^2 $$
with graded characters
%\begin{eqnarray*}
$\chi(e) = \chi(a^2) = 2(1+t)^2$,
$\chi(a) = \chi(b)   = 2(1+t)$,
and $\chi(ab) = 2(1+t^2)$.
%\end{eqnarray*}
Denote with $T$ the trivial irreducible $G$-representation,
and with $X$ and $Y$ certain one-dimensional $G$-representations
with $X\otimes X=Y\otimes Y=T$. Then $\K G$ is isomorphic to
$$ (T\oplus T) \;\oplus\; (T\oplus T\oplus X \oplus Y)\cdot t
\;\oplus\; (T\oplus X\otimes Y)\cdot t^2 $$
as a graded $G$-representation. While the graded characters
factor nicely, the representation itself does not.
\end{example}

%%%%%%%%%%%%%%%%%%%%
%%% Reference on example section, subsections

\section{Graded Traces and Hilbert Series over Nichols Algebras}

In this section we study factorization mechanisms for Nichols algebras, and
hence for their graded traces and especially their Hilbert series.

The root system $\Delta^+$ of a Nichols algebra $\B(M)$, introduced in
Subsection \nref{sec_Preliminaries}, directly presents a factorization of
$\B(M)$ as graded vector space:
$$\bigotimes_{\alpha\in\Delta^+}\B(M_\alpha)
	\stackrel{\sim}{\longrightarrow}\B(M)$$
Note however that over nonabelian groups, the
root system factorization is too crude in general to explain
the full observed factorization into cyclotomic polynomials; especially for the
rank 1 cases in the next section the factorization obtained this way is
trivial.

Nevertheless, we will start in Subsection \nref{sec_Stabilizing} by demonstrating
a factorization of the graded trace of an endomorphism $Q$ that
\emph{stabilizes} a given axiomatized Nichols algebra factorization, such as the
root system above:

\begin{corollary}
  Let $\B(M)$ be a Nichols algebra with factorization $W_\alpha,\alpha\in
  \Delta^+$ and $Q$
  an algebra operator that stabilizes this factorization. Then
  $$\trace{Q}{\B(M)}(t)
    =\prod_{\alpha\in\Delta^+}\trace{Q_\alpha}{\B(M_\alpha)}(t)\,.$$  
\end{corollary}

In Subsection \nref{sec_diagonal} we focus on Nichols algebras over abelian
groups. The preceeding corollary immediately gives an explicit
trace product formula for endomorphisms $Q$ stabilizing the root system, in
terms of cylcotomic polynomials. In particular, it shows the complete
factorization of their Hilbert series.

Using the theory of Lyndon words, we are able to weaken the assumptions on $Q$ to
only \emph{normalize} the root system, i.e.\ acting on it by permutation. We
give such examples where $Q$ interchanges two disconnected
subalgebras in the Nichols algebra, as well as the outer automorphism of a
Nichols algebra of type $A_3$. The authors expect that a more
systematic treatment via root vectors will carry over to endomorphisms
normalizing the root system of a non-abelian Nichols algebra as well.

In Subsection \nref{sec_Nonnormalizing} we present an example of a Nichols
algebra of type $A_2$ and an endomorphism $Q$ induced by its outer automorphism
which fails the normalizing condition on the
non-simple root. Note that in contrast to the $A_3$-example above, there is an
edge flipped by the automorphism, which is called a ``loop'' in literature
(e.g. \cite{Gi06}, p. 47ff).
Nevertheless one observes a factorization of the graded
trace of $Q$, and in this example this can be traced back to a surprising and
appearently new ``symmetrized'' PBW-basis.

In Subsection \nref{sec_Nonabelian} we start approaching Nichols algebras
over nonabelian groups, where one observes astonishingly also
factorization of graded traces into cyclotomic polynomials. This cannot be
explained by the root system alone and might indicate the existence of a finer
root system, which is not at the level of Yetter-Drinfel'd modules.

We can indeed give a family of example constructed as covering Nichols algebras
by the first author (\cite{Len13a}). By construction, these Nichols algebras
possess indeed such a finer root system of different type (e.g.\ $E_6\rightarrow
F_4$). In these examples, the root systems lead to a complete factorization,
but this mechansim does not seem to easily carry over to the general case.

\subsection{Nichols Algebras over Groups}\label{sec_Preliminaries}

The following notions are standard. We summarize them to fix notation and
refer to \cite{HLecture08} for a detailed account.

\begin{definition}
  A \emph{\ydmP}\index{Yetter-Drinfel'd module $M$} $M$ over a group $G$ is a
  $G$-graded vector space over
  $\k$ denoted by layers
  $M=\bigoplus_{g\in G} M_g$
  with a $G$-action on $M$ such that
  $g.M_h=M_{ghg^{-1}}$.
  To exclude trivial cases, we call $M$ \emph{indecomposable} \index{Link
  indecomposable (YDM)}iff the support
  $\{g\;|\;M_g\neq 0\}$ generates all $G$ and \emph{faithful} \index{Faithful
  (YDM)} iff the action is.
\end{definition}

Note that for \emph{abelian} groups, the compatibility condition is just
the stability of the layers $M_g$ under the action of $G$.

The notion of a \ydm automatically brings with it a \emph{braiding} $\tau$ on
$M$---in fact, each group $G$ defines an entire braided category of $G$-\ydms
with graded module homomorphisms as morphisms (e.g. \cite{AG99}, Def. 1.1.15)

\begin{lemma}
  Consider $\tau:M\otimes M\rightarrow M\otimes M$,
  $v\otimes w\mapsto g.w\otimes v\in M_{ghg^{-1}}\otimes M_g$
  for all $v\in M_g$ and $w\in M_h$.
  Then $\tau$ fulfills the \emph{Yang-Baxter-equation}
  $$(\id\otimes \tau)(\tau\otimes \id)(\id\otimes \tau)
    =(\tau\otimes \id)(\id\otimes \tau)(\tau\otimes \id)$$
  turning $M$ into a \emph{braided vector space}.
\end{lemma}

In the non-modular case, the structure of \ydms is well understood
(\cite{AG99} Section 3.1) and can be summarized in the following three lemmata:

\begin{lemma}
  Let $G$ be a finite group and
  let $\k$ be an algebraically closed field whose characteristic does
  not divide $\# G$. Then any finite-dimensional \ydm $M$ over $G$
  is semisimple, i.e.\ decomposes into simple \ydms (the number
  is called \emph{rank} of $M$): $M=\bigoplus_i M_i$.
\end{lemma}

\begin{lemma}
  Let $G$ be a finite group, $g\in G$ arbitrary and $\chi:G\rightarrow\k$
  the character of an irreducible representation $V$ of
  the centralizer subgroup $\Cent(g)=\{h\in G\;|\;gh=hg\}$.
  Define the \ydm $\O_{g}^{\chi}$ to be the induced $G$-representation
  $\k G\ot_{\k\Cent(g)} V$ with $G$-grading ($\k G$-coaction)
  $\delta(h\ot v):=hgh^{-1}\ot(h\ot v)\in\k G\ot(\k G\ot_{\k\Cent(g)} V)$
  for all $h\in G$, $v\in V$.
  It can be constructed as follows:
  \begin{itemize}    
    \item Define the $G$-graduated vector space by
      $$\O_{g}^{\chi}=\bigoplus_{h\in G} \left(\O_{g}^{\chi}\right)_h
      \quad \text{with} \quad
      \left(\O_{g}^{\chi}\right)_h:=\begin{cases}
	V & \text{for $g$-conjugates $h\in [g]$,}\\
	\{0\} & \text{else.}
	\end{cases}$$
    \item Choose a set $S=\{s_1,\ldots s_n\}$ of representatives for the left
      $\Cent(g)$-cosets $G=\bigcup_{k}s_k\Cent(g)$. Then for any $g$-conjugate
      $h\in[g]$ there is precisely one $s_k$ with $h=s_kgs_k^{-1}$.
    \item For the action of any $t\in G$ on any $v_h\in
      \left(\O_{g}^{\chi}\right)_h$ for $h\in[g]$ determine the unique
      $s_i,s_j$, such that 
      $s_igs_i^{-1}=h$ and $s_jgs_j^{-1}=tht^{-1}$.
      Then $s_j^{-1}ts_i\in \Cent(g)$ and using the given $\Cent(g)$-action on
      $V$ we may define
      $t.v_h:=(s_j^{-1}ts_i.v)_{tht^{-1}}$.
  \end{itemize}
  Then $\O_g^\chi$ is simple as \ydm and $\O_g^\chi$ and $\O_{g'}^{\chi'}$
  are isomorphic if and only if $g$ and $g'$ are conjugate and $\chi$ and
  $\chi'$ are isomorphic.
  %To summarize the above: There are $V$-layers $V_g$ for each conjugate of $g$,
  %and elements $h\in \Cent(g)$ act on $V_g$ as prescribed by $V,\chi$, assuming
  %the choice $s_0=1$. On the other hand, the $V_g$ are merely 
  %permuted by the representing elements $h=s_k$. Arbitrary elements $h\in G$ are
  %decomposed uniquely into $h=\Cent(g)s_k$ and act accordingly. 
\end{lemma}

If $\chi$ is the character of a one-dimensional representation $(V,\rho)$, we
will identify $\chi$ with the action $\rho$. In positive characteristic, we
will restrict to $\dim V = 1$, where the character determines its representation.

\begin{lemma}
  Let $G$ be a finite group and let $\k$ be an algebraically closed field
  whose characteristic does not divide $\# G$. 
  Then any simple \ydm $M$ over $G$ is isomorphic to some $\O_{g}^{\chi}$
  for some $g\in G$ and a character $\chi:G\rightarrow \k$ of an irreducible
  representation $V$ of the centralizer subgroup $\Cent(g)$.
\end{lemma}

\begin{example}
For finite and \emph{abelian} $G$ over algebraically closed $\k$ with
$\cha K\nmid \#G$ we have 1-dimensional
simple \ydms $M_i=\O_{g_i}^{\chi_i}=x_i\k$ and hence the braiding is
\emph{diagonal} (i.e.\ $\tau(x_i\otimes x_j)= q_{ij}(x_j\otimes x_i)$)
with \emph{braiding matrix} $q_{ij}:=\chi_j(g_i)$.
\end{example}

\begin{definition} \label{DEF_skew_derivations}
  Consider the tensor algebra $\T M$, i.e. for any homogeneous basis $x_i\in
  M_{g_i}$ the algebra of words in all $x_i$. We may uniquely define \emph{skew
  derivations} on this algebra, i.e. maps $\partial_i:\;\T M\rightarrow \T M$
  by
%  \begin{itemize}
%    \item $\partial_i(1)=0$
%    \item $\partial_i(x_j)=\delta_{ij}1$
%    \item $\partial_i(ab)=\partial_i(a)b+(g_i.a)\partial_i(b)$
%  \end{itemize}
  $\partial_i(1)=0$, $\partial_i(x_j)=\delta_{ij}1$, and
  $\partial_i(ab)=\partial_i(a)b+(g_i.a)\partial_i(b)$.
\end{definition}

\begin{definition}\label{page_Nichols}
  The \emph{Nichols algebra}\index{Nichols algebra $\B(M)$} $\B(M)$ is the
  quotient of $\T M$ by the largest
  homogeneous ideal $\mathfrak{I}$ invariant under all $\partial_i$,
  such that $M\cap \mathfrak{I}=\{0\}$.
\end{definition}

In specific instances, the Nichols algebra may be finite-dimensional. 
This is a remarkable phenomenon (and the direct reason for the
finite-dimensional truncations of $\U_q(\g)$ for $q$ a root of unity):

\begin{example}
  Take $G=\Z_2$ and $M=M_e\oplus M_g$ the \ydm with dimensions 
  $0+1$ i.e. $q_{11}=-1$, then $x^2\in \mathfrak{I}$ and hence
  the Nichols algebra $\B(M)=\k[x]/(x^2)$ has dimension $2$.
\end{example}
More generally a 1-dimensional \ydm with $q_{ii}\in\k_n$ a primitive $n$-th root
of unity has Nichols algebra $\B(M)=k[x]/(x^n)$.
\begin{example}
  Take $G=\Z_2$ and $M=M_e\oplus M_g$ the \ydm with dimensions 
  $0+2$ i.e. $q_{11}=q_{22}=q_{12}=q_{21}=-1$  then 
  $$\B(M)=\k\langle x,y\rangle/(x^2,y^2, xy+yx)=\bigwedge M$$
\end{example}
In the abelian case, Heckenberger (e.g \cite{Heck09}) introduced $q$-decorated
diagrams\index{q@$q$-decorated diagram (YDM)}\label{qdiagram}, with each node
corresponding to a simple \ydm decorated by $q_{ii}$,
and each edge decorated by $\tau^2=q_{ij}q_{ji}$ and edges are drawn if the
decoration is $\neq 1$; it turns out that this data is all that is needed to determine
the respective Nichols algebra.

\begin{theorem}\cite{HS10a}
% No restrictions on \k needed.
  Let $\B(M)$ be a Nichols algebra of finite dimension over an arbitrary group
  $G$, then there exists a rootsystem $\Delta\subset \Z^N$ with positive roots
  $\Delta^+$ and a truncated basis of monomials in $x_\alpha\in
  \B(M)_{|\alpha|}$. Namely, the multiplication in  $\B(M)$ is an isomorphism of
  graded vectorspaces
  $\B(M) \cong \bigotimes_{\alpha\in\Delta^+}\B(M_\alpha)$.
\end{theorem}

\begin{example}
  Assume $q_{11}=q_{22}=q_{12}q_{21}=-1$, then the diagram is:
  \image{0.2}{Rank2A2Decorated}
  Some calculations show that $x_3:=[x_1,x_2]_\tau\neq 0$ ($\not\in M$!),
  but $[x_2,[x_1,x_2]_\tau]_\tau=[x_1,[x_1,x_2]_\tau]_\tau=0$.
  Hence $\B(M)$ corresponds to the Borel part of $A_2=\mathfrak{sl}_3$.

  As $q_{11}=q_{22}=-1$ as well as $q_{33}=(q_{11}q_{12})(q_{22}q_{21})=-1$,
  all three  Nichols algebras $\B(x_i\k)$ of rank 1 are 2-dimensional. The
  Nichols algebra $\B(x_1\k\oplus x_2\k)$ itself is 8-dimensional with PBW-Basis
  $x_1^ix_2^jx_3^k$ with $i,j,k\in\{0,1\}$, i.e. multiplication in  $\B(M)$
  yields a vector space bijection:
  $$\B(M)\cong \B(x_1\k)\otimes \B(x_2\k)\otimes
    \B(x_3\k)=\k[x_1]/(x_1^2)\otimes \k[x_2]/(x_2^2)\otimes \k[x_3]/(x_3^2)$$
\end{example}
In the same sense, over abelian $G$ for $a_{ij}$ any proper Cartan matrix of a
semisimple Lie algebra is realized for braiding matrix
$q_{ij}q_{ji}=q_{ii}^{-a_{ij}}$.

However, several additional exotic examples of
finite-dimensional Nichols algebras exist, that possess unfamiliar Dynkin
diagrams, such as a multiply-laced triangle, and where Weyl reflections may
connect different diagrams (yielding a \emph{Weyl groupoid}). Heckenberger
completely classified all Nichols algebras over abelian $G$ in \cite{Heck09}.

Over \emph{nonabelian} groups, still much is open, but progress is made:
Andruskiewitsch, Heckenberger, and Schneider studied the Weyl groupoid in this
setting as well and established a
root system and a PBW-basis for finite-dimensional Nichols algebras in
\cite{AHS10}.
\begin{itemize}
 \item By detecting certain ``defect'' subconfigurations (so-called
type $D$) most higher symmetric and all alternating groups and later many
especially sporadic groups were totally discarded
(Andruskiewitsch et. al. \cite{AZ07},\cite{AFGV11} etc.).
 \item Only few finite-dimensional indecomposable examples are known so
far, namely $\D_4$ of type $A_2$ and $\S_3,\S_4,\S_5$ of type $A_1$ (Schneider
et. al. \cite{MS00}), higher analogues of $\D_4$ (\cite{HS10b}), and some
rank 1 examples over metacyclic groups \cite{GranaZoo}, as well as infinite
families of Lie type over extraspecial groups \cite{Len13a}.
  \item In \cite{HV13} all finite-dimensional Nichols algebras of Rank 2 have
  been determined
\begin{itemize}
 \item Type $A_2$ over a quotient of the group $\Gamma_2$, which has been known 
  (e.g. $\D_4$)
 \item Type $G_2$ over a rack of type $T$, between a central element
  and a group $\mathbb{A}_4$.
 \item Three Nichols algebras over quotients of the group $\Gamma_3$ (e.g.
  $\S_3$), not of finite Cartan type and depending
  on the characteristic of the base field
\end{itemize}
\end{itemize}

\subsection{A First Trace Product Formula}\label{sec_Stabilizing}

In the following we want to use the root system $\Delta^+$ of a Nichols algebra
$\B(M)$ to derive a trace product formula.

\begin{definition}
  A \emph{factorization} of a Nichols algebra $\B(M)$ of a braided vectorspace
  is a collection of braided vector spaces $(M_\alpha)_{\alpha\in\Delta^+}$
  with some arbitrary index set $\Delta^+$, such that:
  \begin{itemize}
   \item All $M_\alpha$ are grading-homogeneous, braided subspaces of the
      Nichols algebra $M_\alpha\subset\B(M)$.
   \item The multiplication in $\B(M)$ induces a
      graded isomorphism of braided vector spaces
      $\mu_{\B(M)}:\;\bigotimes_{\alpha\in\Delta^+}\B(M_\alpha)
	\stackrel{\sim}{\longrightarrow}\B(M)$.
  \end{itemize}

  An operator $Q$ on $\B(M)$ is said to \emph{stabilize} the factorization iff
  $QM_\alpha\subset M_\alpha$ for all $\alpha\in\Delta^+$. In this case we
  denote by $Q_\alpha:=Q|_{\B(M_\alpha)}$ the restriction.
  $Q$ is said to \emph{normalize} the factorization iff for each
  $\alpha\in\Delta^+$ there is a $\beta\in\Delta^+$ with 
  $QM_\alpha\subset M_\beta$. In this case $Q$ acts on $\Delta^+$ by
  permutations and we denote the shifting restriction
  $Q_{\alpha\rightarrow\beta}:=Q|_{M_\alpha}$ with image in
  $M_\beta=M_{Q\alpha}$.
\end{definition}
\begin{example}
The root system $W_\alpha,\;\alpha\in\Delta^+$ of a Nichols algebra
$\B(M)$ of a \ydm $M$ (see section \nref{sec_Preliminaries}) is the leading
example of a factorization. 
\end{example}

A factorization of a Nichols algebra can be used to derive a product formula of
the trace and graded trace of some operator $Q$. A first application is:

\begin{lemma}
  Let $\B(M)$ be a Nichols with factorization $M_\alpha,\alpha\in \Delta^+$ and
  $Q$ an algebra operator that stabilizes this factorization. Then the trace
  of $Q$ is
  $\tr\left(Q\right)
    =\prod_{\alpha\in\Delta^+}\tr\left(Q_\alpha\right)$.
\end{lemma}
\begin{proof}
  We evaluate the trace in the provided factorization: Let $Q$ act diagonally on
  the tensor product $\bigotimes_{\alpha\in\Delta^+}\B(M_\alpha)$ by acting on
  each factor via the restriction $Q_\alpha$, which is possible because $Q$ was
  assumed to stabilize this factorization. The action of an algebra operator $Q$
  commutes with the multiplication $\mu_{\B(M)}$ so the trace of $Q$ acting
  on both sides coincides. The trace on a tensor product is the
  product of the respective trace and hence we get
    $$\tr\left(Q|_{\B(M)}\right)
      \,=\,\tr\left(Q|_{\bigotimes_{\alpha\in\Delta^+}\B(M_\alpha)}\right)
      \,=\prod_{\alpha\in\Delta^+}\!\tr\left(Q|_{\B(M_\alpha)}\right)
      \,=\prod_{\alpha\in\Delta^+}\!\tr\left(Q|_\alpha\right). $$
\end{proof}

To calculate the graded trace with the preceeding lemma, first note that $t^E$
is a graded algebra automorphism, so if $Q$ fulfills the conditions of the
lemma, so does $Y=t^EQ$. Hence, we find $\trace{Q}{\B(M)}=\tr(t^EQ)=\tr(Y)$: 

\begin{corollary}\label{cor_ProductFormula}
  Let $\B(M)$ be a Nichols algebra with factorization $W_\alpha,\alpha\in
  \Delta^+$ and $Q$
  an algebra operator that stabilizes this factorization. Then
  $$\trace{Q}{\B(M)}(t)
    =\prod_{\alpha\in\Delta^+}\trace{Q_\alpha}{\B(M_\alpha)}(t)\,.$$  
\end{corollary}
\begin{example}
Let the braided vector space $M=x_1\k\oplus x_2\k$ be defined by
$q_{ij}=
\begin{pmatrix}
  -1 & -1\\
  +1 & -1\\
\end{pmatrix}$.
Then the diagonal Nichols algebra $\B(M)$ is of standard Cartan type $A_2$ and
possesses a factorization $\Delta^+=\{\alpha_1,\alpha_2,\alpha_{12}\}$ with
$$M_{\alpha_1}=x_1\k,
\quad M_{\alpha_2}=x_2\k,
\quad M_{\alpha_{12}}=x_{12}\k,
\quad x_{12}:=[x_1,x_2]_q:=x_1x_2+x_2x_1\,.$$
All braidings are $-1$, hence $\B(x_\alpha)\cong \k[x_\alpha]/(x_\alpha^2)$.
This implies that the multiplication in $\B(M)$ is an isomorphism of graded
vector spaces:
$$\mu_{\B(M)}:\;  
\k[x_{{1}}]/(x_{{1}}^2)
\otimes \k[x_{{2}}]/(x_{{2}}^2)
\otimes \k[x_{{12}}]/(x_{{12}}^2)
\cong \B(M)$$
This shows that the Hilbert series  $\H(t)=\trace{1}{\B(M)}(t)$ is
\begin{align*}
  \trace{1}{\B(M)}(t)
  \;&=\;\trace{1}{\B(M_{\alpha_{1}})}\trace{1}{\B(M_{\alpha_{2}})}
    \trace{1}{\B(M_{\alpha_{12}})}
  \;=\;(1+t)(1+t)(1+t^2)\\
  \;&=\;1+2t+2t^2+2t^3+t^4
\end{align*}
\end{example}
\begin{example}
In the previous example of a Nichols algebra $\B(M)$, let $Q\in \End(M)$ be
defined by 
$Qx_1:=x_1$ and $Qx_2:=-x_2$.
This map preserves the braiding and hence extends uniquely to an algebra
automorphism on $\B(M)$, in particular,
$Qx_{12}=Q(x_1x_2+x_2x_1)=-(x_1x_2+x_2x_1)=-x_{12}$ holds.
A direct calculation yields:
\begin{align*}
\trace{Q}{\B(M)}(t)
  \;&=\;\tr\left(Q|_{1_{\B(M)}}\right)
  +t\cdot \tr\left(Q|_{x_1,x_2}\right)
  +t^2\cdot \tr\left(Q|_{x_1x_2,x_{12}}\right) \\
  &\qquad+t^3\cdot \tr\left(Q|_{x_1x_{12},x_2x_{12}}\right)
  +t^4\cdot \tr\left(Q|_{x_1x_2x_{12}}\right) \\
  \;&=\;\tr\begin{pmatrix} 1 \end{pmatrix}
  +t\cdot\tr\begin{pmatrix}  1 & 0 \\ 0 & -1 \end{pmatrix}
  +t^2\cdot\tr\begin{pmatrix} -1 & 0 \\ 0  & -1 \end{pmatrix}\\
  &\qquad+t^3\cdot\tr\begin{pmatrix}  1 & 0 \\ 0 & -1 \end{pmatrix}
  +t^4\tr\begin{pmatrix}  1  \end{pmatrix}\\
  \;&=\;1-2t^2+t^4
\end{align*} 
The product formula returns for the same trace:
$$\trace{Q}{\B(M)}(t)
  \;=\;\trace{Q}{\B(M_{\alpha_{1}})}(t)\cdot\trace{Q}{\B(M_{\alpha_{2}})}(t)\cdot
    \trace{Q}{\B(M_{\alpha_{12}})}(t)
  \;=\;(1+t)(1-t)(1-t^2)$$
\end{example}

In the next section, we study the special case of a diagonal Nichols
algebra, we will also study examples of operators $Q$, that do neither
stabilize nor normalize the root system. However, their graded trace is still
factorizing, which indicates the existence of alternative PBW-basis. We will
construct such for the case $N=2,\;q=\sqrt{-1}$.

\subsection{Nichols Algebra Over Abelian Groups}\label{sec_diagonal}
We now restrict our attention to
the Nichols algebra $\B(M)$ of a \ydm $M$ of rank $n$ over an abelian
group $G$ and $\k=\C$. This means $M$ is diagonal i.e. the sum of $1$-dimensional braided
vector spaces $x_i\k$. According to \cite{Heck09}, $\B(M)$ possesses an
arithmetic root system $\Delta^+$ that can be identified with a set of Lyndon
words $\mathcal{L}$ in $n$
letters, with word length corresonding to the grading in the Nichols algebra.
Such a Lyndon word corresponds to iterated $q$-commutators in the letters
$x_i$ according to iterated Shirshov decomposition of the word.

For any positive root $\alpha\in\Delta^+$ we denote by $N_\alpha$ the
order of the self-braiding $\chi(\alpha,\alpha)=q_{\alpha,\alpha}$. It is known
that this determines 
$\B(x_\alpha)=\k[x_\alpha]/(x_\alpha^{N_\alpha})$
and we denote by $|\alpha|$ the length of the Lyndon word resp.\ the degree of
the root vector $x_\alpha$ in the Nichols algebra grading.

We further denote by $g_\alpha$ the $G$-grading of $x_\alpha$, extending the
$G$-grading of $M$ on simple roots. Moreover,
we denote the scalar action of any $g\in G$ on $x_\alpha$ by
$\alpha(g)=\chi_M(g_\alpha,g)\in\k^\times$, extending the $G$-action of $M$ on
simple roots. We frequently denote 
$(N)_t=1+t+\cdots t^{N-1}=\frac{1-t^N}{1-t}$.

\begin{lemma}\label{lm_ProductFormula_abelian}
  Let $Q$ be an algebra operator on a diagonal Nichols algebra $\B(M)$ that
  stabilizes the root system, i.e. for all $\alpha\in\Delta^+$ the root vector
  $x_\alpha$ is an eigenvector to $Q$ with eigenvalue $\lambda_\alpha\in\k$.
  Then the product formula of Corollary \nref{cor_ProductFormula} reads as
  follows:
  $$\trace{Q}{\B(M)}(t)
    =\prod_{\alpha\in\Delta^+}
      (N_\alpha)_{\lambda_\alpha t^{|\alpha|}}$$
  Especially for the action of a group element $g\in G$ we get 
  $$\trace{g}{\B(M)}(t)
    =\prod_{\alpha\in\Delta^+}
      (N_\alpha)_{\alpha(g)t^{|\alpha|}}\,.$$
\end{lemma}
\begin{proof}
  The factorization follows from Corollary \nref{cor_ProductFormula}. We yet
  have to verify the formula on each factor
  $W_\alpha=\B(x_\alpha)=\k[x_\alpha]/(x_\alpha^{N_\alpha})$.
  $W_\alpha$ has a basis $x_\alpha^k$ for $0\leq k< N_\alpha$ with degrees
  $k|\alpha|$. By assumption $Q$ acts on
  $x_\alpha$ via the scalar $\lambda_\alpha$ and by multiplicativity on
  $x_\alpha^k$ by $\lambda_\alpha^k$. Altogether:
  $$ \trace{Q}{\B(x_\alpha)}
    \;=\;\sum_{k=0}^{N_\alpha-1} \lambda_\alpha^k \cdot t^{k|\alpha|}
    \;=\;\sum_{k=0}^{N_\alpha-1} (\lambda_\alpha\cdot t^{|\alpha|})^k
    \;=\;(N_\alpha)_{\lambda_\alpha|\alpha|}\,. $$
\end{proof}

\begin{example}
Let the braided vector space $M=x_1\k\oplus x_2\k$ be defined by
$q_{ij}=
\begin{pmatrix}
  q^2 & q^{-1}\\
  q^{-1} & q^2\\
\end{pmatrix}$
with $q$ a primitive $2N$-th root of unity. Then the diagonal Nichols algebra
$\B(M)$ is of standard Cartan type $A_2$ and
possesses a factorization $\Delta^+=\{\alpha_1,\alpha_2,\alpha_{12}\}$ with
$M_{\alpha_1}=x_1\k$,
$M_{\alpha_2}=x_2\k$, and
$M_{\alpha_{12}}=x_{12}\k$, where
$x_{12}:=[x_1,x_2]_q:=x_1x_2-q^{-1}x_2x_1$.
This implies that the multiplication in $\B(M)$ is an isomorphism of graded
vector spaces:
\begin{align*}
\mu_{\B(M)}:\;\B(M)
\;&\cong\; \B(M_{\alpha_1})\otimes \B(M_{\alpha_2})
\otimes\B\left(M_{\alpha_{12}}\right)\\
\;&\cong\; \k[x_1]/(x_{1}^N)\otimes \k[x_2]/(x_{2}^N)\otimes \k[x_{12}]/(x_{12}^N)\,.
\end{align*}
This obviously agrees with the Hilbert series in Lemma \nref{lm_ProductFormula_abelian}:
$$ \H(t) \;=\; \trace{1_\B(M)}{\B(M)}(t)
\;=\; \prod_{i=1}^3\trace{1_\B(M_{\alpha_i})}{\B(M_{\alpha_i})}(t)
\;=\; (N)_t(N)_t(N)_{t^2}\,. $$
Let us now apply the formula of
Lemma \nref{lm_ProductFormula_abelian} to calculate the graded trace of the
action of group elements (which stabilize the root system): We realize the
braided vector space $M$ as a Yetter-Drinfel'd module over $G:=\Z_{2N}\times
\Z_{2N}=\langle g_1,g_2\rangle$, 
such that $x_1$ is $g_1$-graded and $x_2$ is $g_2$ graded, with suitable
actions
$$g_1x_1=q^2x_1\qquad g_1x_1=q^2x_1 
\qquad g_1x_2=q^{-1}x_2\qquad g_2x_1=q^{-1}x_1\,.$$
Then we get for the action of each group element $g_k$:
$$ \trace{g_k}{\B(M)}(t)
  \;=\;\prod_{\alpha\in\{\alpha_1,\alpha_2,\alpha_{12}\}}
    \trace{g_k}{\B(M_\alpha)}
  \;=\;(N)_{q^2t}(N)_{q^{-1}t}(N)_{qt^2}
$$
\end{example}

We now look at graded traces of automorphisms $Q$, where $Q$ does not stabilize
the root system. We restrict ourselves to diagonal Nichols
algebras $\B(M)$, so we may use the theory of Lyndon words.

The PBW-basis consists of monotonic monomials
$$[u_1]^{n_1}[u_2]^{n_2}\cdots [u_k]^{n_k}\quad
\text{with}\quad u_1\,>\,u_2\,>\,\cdots \,>\,u_k
\qquad \forall u_i\in\mathcal{L}\,.$$
The PBW-basis carries the lexicographic order $<$ and by
\cite{HLecture08} (Remark after Thm 3.5) this is the same as the
lexicographic order of the composed words
$u_1^{n_1}u_2^{n_2},\cdots u_k^{n_k}$. For a sequence of Lyndon words
$\vec{u}=(u_1,\ldots u_k)$, not necessarily monotonically sorted, we define 
$[\vec{u}]:=[u_1][u_2]\cdots [u_k]$.
In particular, sorted sequences correspond to the PBW-basis. For any sequence of
Lyndon words $\vec{u}$, denote by $\vec{u}^{\;\sort}$ its monotonic
sorting.

\begin{lemma}
\begin{enumerate}
 \item[a)] For any sequence of Lyndon words $\vec{u}$, not necessarily monotonically
  sorted, we have $[\vec{u}]\,=\,q\cdot [\vec{u}^{\;\sort}]\,+\,\text{smaller}$,
  where $q\neq 0$ and ``smaller'' denotes linear combinations of PBW-elements
  lexicographically smaller than the PBW-element $[\vec{u}^{\;\sort}]$.
 \item[b)] Let $\sigma\in\S_k$ and $u_1>u_2>\cdots >u_k\in\mathcal{L}$ then 
\begin{eqnarray*}
  [u_{\sigma(1)}]^{n_{\sigma(1)}}[u_{\sigma(2)}]^{n_{\sigma(2)}}\cdots
  [u_{\sigma(k)}]^{n_{\sigma(k)}}
  \;=\;q_{\sigma}\cdot [u_1]^{n_1}[u_2]^{n_2}\cdots
  [u_k]^{n_k} \quad \\ \,+\,\text{smaller}\,,
\end{eqnarray*}
  where $q_{\sigma}$ is the scalar factor associated to the braid group
  element $\hat{\sigma}\in\BB_k$, which is the image of $\sigma$
  under the Matsumoto section.
    Explicitely $q_{(i,i+1)}\,=\,\chi(\deg(x_{u_i}),\deg(x_{u_{i+1}}))^{n_i n_{i+1}}$
    in the notation of \cite{HLecture08} and general $q_\sigma$ are
    obtained by multiplying such factors along a reduced expression of $\sigma$.
\end{enumerate}
\end{lemma}
\begin{proof}
  {\em Claim a):} We perform induction on the multiplicity of
  the highest appearing Lyndon word: So for a fixed $w\in\mathcal{L},N\in\N$
  suppose that the claim has been proven for all sequences $\vec{u'}$ with
  $u'_i\leq w$ and strictly less then $N$ indices $i$ with $u'_i=w$.

  Consider then a sequence $\vec{u}$ with $u_i\leq w$ and precisely $N$ indices
  $i$ with $u_i=w$. We perform a second induction on the index $i$ of the
  leftmost appearing $w=u_i$:
  \begin{itemize}
    \item If $w=u_1$ we may consider the
  sequence $\vec{u}_{\hat{1}}:=(u_2,\ldots u_k)$ having strictly less
  $w$-multiplicity. By induction hypothesis,
  $[\vec{u}_{\hat{1}}]=[\vec{u}_{\hat{1}}^{\;\sort}]+\text{smaller}$.
  Since $\vec{u}^{\;\sort}=(w,\vec{u}_{\hat{1}}^{\;\sort})$
  the assertion then also holds for $\vec{u}$.
   \item Otherwise, let $u_{i+1}=w$ be the leftmost appearance of $w$,
  especially $u_i<u_{i+1}=w$. By \cite{HLecture08} Prop.\ 3.9 we
  then have 
  $[u_i][u_{i+1}]=q\cdot[u_{i+1}][u_i]+\text{smaller}$
  where $q\neq 0$ and ``smaller'' means products of Lyndon words $[v_l]$ with
  $u_i<v<u_{i+1}=w$. Thus all products $[u_1]\cdots\text{smaller}\cdots [u_k]$
  contain $w$ with a multiplicity less then $\vec{u}$; by induction hypothesis
  these are linear combination of PBW-elements lexicographically strictly
  smaller then  $[\vec{u}^{\;\sort}]$. The remaining summand 
  $[u_1]\cdots[u_{i+1}][u_i]\cdots [u_k]$ has $w=u_{i+1}$ in a leftmore
  position and the claim follows by the second induction hypothesis.
  \end{itemize}

  {\em Claim b):} 
  We proceed by induction on the length of $\sigma\in\S_k$, which is the length
  of any reduced expression for $\sigma$. For $\sigma=\id$ we're done, so assume
  for some $i$ that $u_{\sigma(i)}<u_{\sigma(i+1)}$, hence
  $\sigma=(\sigma(i),\sigma(i+1))\sigma'$ with $\sigma'$ shorter.

  Again by \cite{HLecture08} Prop. 3.9 we have:
  $$ [u_{\sigma(i)}][u_{\sigma(i+1)}]
    \;=\;\chi(\deg(x_{u_i}),\deg(x_{u_{i+1}}))
    \cdot[u_{\sigma(i+1)}][u_{\sigma(i)}]\,+\,\text{smaller}\,. $$
  Moreover, for any sequences of Lyndon words $\vec{a},\vec{b}$, claim
  a) proves that $[\vec{a}]\cdot \text{smaller}\cdot [\vec{b}]$ is a linear-combination
  of PBW-elements lexicographically smaller than $[\vec{u}]$. Hence
  inductively 
  \begin{align*}
  [u_{\sigma(i)}]^{n_{\sigma(i)}}[u_{\sigma(i+1)}]^{n_{\sigma(i+1)}}
  \;&=\;\chi(\deg(x_{u_i}),\deg(x_{u_{i+1}}))^{n_{\sigma(i)}n_{\sigma(i+1)}}\\
  &\qquad\cdot[u_{\sigma(i+1)}]^{n_{\sigma(i+1)}}[u_{\sigma(i)}]^{n_{\sigma(i)}}
  \,+\,\text{smaller}\,.
  \end{align*}
  By the same argument (again using claim a) we may multiply both sides with
  the remaining factors:
  \begin{align*}
  & [u_1]^{n_1}\cdots[u_{\sigma(i)}]^{n_{\sigma(i)}}
  [u_{\sigma(i+1)}]^{n_{\sigma(i+1)}}\cdots [u_k]^{n_k} \\
  & \quad \;=\;\chi(\deg(x_{u_i}),\deg(x_{u_{i+1}}))^{n_{\sigma(i)}n_{\sigma(i+1)}} \\
  & \qquad \qquad \cdot[u_1]^{n_1}\cdots[u_{\sigma(i+1)}]^{n_{\sigma(i+1)}}
  [u_{\sigma(i)}]^{n_{\sigma(i)}}\cdots[u_k]^{n_k}
  \,+\,\text{smaller}\,.
  \end{align*}
  We may now use the induction hypothesis on $\sigma'$, which is shorter.
\end{proof}

% For any positive root $\alpha\in\Delta^+$ we denote by $g_\alpha$ the
% associated $G$-grading of $x_\alpha$, extending the $G$-grading of $M$ on simple
% roots. Moreover we denote the scalar action of any $g\in G$ on $x_\alpha$ by 
% $\alpha(g)=\chi_M(g_\alpha,g)\in\k^\times$, extending the $G$-action of $M$ on
% simpe roots.\\
\begin{theorem}
  Let $\B(M)$ be a finite-dimensional Nichols algebra over a Yetter-Drinfel'd
  module $M$ over an abelian group $G$. Let $Q$ be an automorphism of the graded
  algebra $V=\B(M)$ permuting the
  roots $QV_\alpha=V_{Q\alpha}$ and denote the action on root vectors by
  $$Qx_\alpha=:\lambda_Q(\alpha)x_{Q\alpha}
    \qquad \lambda_Q:\;\Delta^+\rightarrow\k\,.$$
  On any orbit $A\in\O_Q\left(\Delta^+\right)$ all orders $n_\alpha$ of
  $x_\alpha$ coincide for $\alpha\in A$ and we denote this value by $N_A$.
  Similarly, all degrees $\alpha$ coincide and we denote the sum over the
  orbit in slight abuse of notation $|A|=|\alpha|\cdot\#A$. Then
  $$\trace{Q}{V_A}(t)
    \;=\;\prod_{A\in\O_Q\left(\Delta^+\right)}(N_A)_{q_A(Q)\lambda_Q(A)t^{|A|}}$$
  with the $q$-symbol $(N)_t:=1+t+\cdots+t^{N-1}=\frac{1-t^N}{1-t}$ and
  $q(Q)\in\k^\times$ the scalar braiding factor of $Q$ acting as an
  element of $\BB_{|A|}$ on $A$, as in the preceeding lemma.
\end{theorem}
\begin{proof}
We start with the factorization along the rootsystem
$$\trace{Q}{V}(t)\;=\prod_{A\in\O_Q\left(\Delta^+\right)}\trace{Q}{V_A}(t)
    \qquad V_A\;:=\;\bigotimes_{\alpha\in A}V_{\alpha}$$
with $Qx_\alpha=:\lambda_Q(\alpha)x_{Q\alpha}$ as assumed. The action of $Q$
on monomials $\otimes_{\alpha\in A}x_\alpha^{k_\alpha}\in V_A$ can be calculated
using the previous lemma:
\begin{align*}
Q\bigotimes_{\alpha\in A}x_\alpha^{k_\alpha}
  \;&=\;\bigotimes_{\alpha\in A}\lambda_Q(\alpha)^{k_\alpha}
    x_{Q\alpha}^{k_\alpha}
  \;=\;\prod_{\alpha\in A}\lambda_Q(\alpha)^{k_\alpha}
  \cdot\bigotimes_{\alpha\in A} x_{Q\alpha}^{k_\alpha} \\
  \;&=\;q_A(Q)^k\bigotimes_{\alpha\in A}x_\alpha^{k_\alpha}\,+\,\text{smaller}\,,
\end{align*}
where $q(Q)\in\k^\times$ denotes the scalar braiding factor of $Q$ acting as an
element of $\BB_{|A|}$ on $A$. The trace over $\trace{Q}{V_A}$ may be evaluated
in this monomial basis and the only contributions come from monomials with all
$k_\alpha$ equal:
$$Q\bigotimes_{\alpha\in A}x_\alpha^{k}
  \;=\;q_A(Q)^k\prod_{\alpha\in A}\lambda_Q(\alpha)^{k}
  \cdot\bigotimes_{\alpha\in A} x_{Q\alpha}^{k}
  \;=\;\lambda_Q(A)^k\cdot \bigotimes_{\alpha\in A} x_{Q\alpha}^{k}\,.$$
Thus we can calculate the trace in terms of $q$-symbols: We sum up the scalar
action factors
$q_A(Q)^k\lambda_Q(A)^k$ on all $Q$-fixed basis elements $\bigotimes_{\alpha\in
A}x_{Q\alpha}^k$ for $k<N_A$ and multiply by the level:
$$\trace{Q}{V_A}(t)\;=\;\sum_{k=0}^{N_A-1} q_A(Q)^k\lambda_Q(A)^k
  \cdot t^{k\sum_{\alpha\in A}|\alpha|}
  \;=\;\sum_{k=0}^{N_A-1}\big(q_A(Q)\lambda_Q(A)\big)^k t^{k|A|}$$
for $N_A:=N_\alpha$ and $|A|:=|\alpha|\cdot\#A$ as in the assertion independent
of $\alpha$. Therefore,
  \begin{align*}
    \trace{Q}{V_A}(t)
    \;&=\,\sum_{k=0}^{N_A-1} \left(q_A(Q)\lambda_Q(A)t^{|A|}\right)^k
    \;=\;(N_A)_{q_A(Q)\lambda_Q(A)t^{|A|}}\\
    \trace{Q}{V}(t)
    \;&=\prod_{A\in\O_Q\left(\Delta^+\right)}
      \trace{Q}{V_A}(t)
    \;=\prod_{A\in\O_Q\left(\Delta^+\right)}
      (N_A)_{q_A(Q)\lambda_Q(A)t^{|A|}}
  \end{align*}
\end{proof}

We now give examples where this formula can be applied. Note that the
normalizing-condition seems to be very restrictive and the examples below
crucially rely on exceptional behaviour for $q=-1$. Nevertheless, we obtain
nontrivial examples, such as $A_3^{q=-1}$, and will use the previous formula
systematically in the last subsection in conjunction with the finer root system
presented in the examples there.

\begin{example}
 Let $q$ be a primitive $2N$-th root of unity and $M$ a Yetter-Drinfel'd
 module with braiding matrix corresponding to $A_1\cup A_1$ and
 $q_{ij}=\begin{pmatrix} q^2 & -1 \\ -1 & q^2 \end{pmatrix}$.
 Extend $Q:x_1\leftrightarrow x_2$ to an algebra automorphism of
 $\B(M)$. Note that any other off-diagonal entries $a,a^{-1}$ would let $Q$
 fail to preserve the braiding matrix. We have
$\B(M)\cong\k[x_1]/(x_1^N)\otimes\k[x_1]/(x_1^N)$ and explicitely  calculate
 $Qx_1^ix_2^j=x_2^ix_1^j=(-1)^{ij}x_1^jx_2^i$.
 Hence all contributions to the graded trace are balanced monomials with $i=j$,
 yielding
$$ \trace{Q}{\B(M)}(t) \;=\; \sum_{i=0}^{N-1} (-1)^{i^2}t^{2i}
  \;=\;\sum_{i=0}^{N-1} (-1)^{i}t^{2i}
  \;=\;(N)_{-t^2}\,. $$
\end{example}

\begin{example}\label{exm:A2}
Consider again the $A_2$ example $M=x_1\k\oplus x_2\k$ with
$q_{ij}=
\begin{pmatrix}
  q^2 & q^{-1}\\
  q^{-1} & q^2\\
\end{pmatrix}$
and $q$ a primitive $2N$-th root of unity. Consider in this case the
diagram automorphism $Q:x_1\leftrightarrow x_2$ again. In the theory of Lie algebra
foldings, the flipped edge $x_1x_2$ is called a \emph{loop}. We easily
calculate that in such cases the standard root system is never normalized by
$Q$, since $q\neq -1$:
$ Qx_{12}
 =Q(x_1x_2-q^{-1}x_2x_1)
 =x_2x_1-q^{-1}x_1x_2
 =-q^{-1}(x_1x_2-qx_2x_1)\neq x_{12}$.
We will discuss this example and its factorization in the next
subsection.
\end{example}

\begin{example}
  Consider the braiding matrix 
$$q_{ij}=
\begin{pmatrix}
  -1 & i & -1\\
  i & -1 & i\\
  -1 & i & -1
\end{pmatrix}$$
which gives rise to a Nichols algebra of type $A_3$ and hence a root system
\begin{align*}
 \B(M)\;&\cong\; \k[x_{1}]/(x_{1}^2)\,\otimes\,\k[x_{2}]/(x_{2}^2)\,\otimes\,
 \k[x_{3}]/(x_{3}^2)\\
  &\qquad \,\otimes\,\k[x_{12}]/(x_{12}^2)\,\otimes\,
 \k[x_{32}]/(x_{32}^2)\,\otimes\,\k[x_{1(32)}]/(x_{1(32)}^2)
\end{align*}
(these are choices) with 
$$x_{12}=x_1x_2-ix_2x_1\qquad x_{32}=x_3x_2-ix_2x_3
\qquad x_{123}=x_1x_{32}+ix_{32}x_1\,.$$
Notice that for the specific choice of $q^2=-1$, by chance, we also have
\begin{align*}
 x_{3(12)}
  \;&:=\;x_3x_{12}+ix_{12}x_3
  \;=\;x_3x_1x_2-ix_3x_2x_1+ix_1x_2x_3+x_2x_1x_3\\
  \;&=\;-x_1x_3x_2-ix_3x_2x_1+ix_1x_2x_3-x_2x_3x_1\\
  \;&=\;-\left(x_1x_3x_2-ix_1x_2x_3+ix_3x_2x_1+x_2x_3x_1\right)
  \;=\;-x_{1(32)}\,.
\end{align*}
Consider the diagram automorphism $Q:x_1\leftrightarrow x_3$ that preserves the
braiding matrix and hence gives rise to an algebra automorphism of $\B(M)$. We
show that it normalizes the chosen factorization:
\begin{align*}
  Qx_{12}
  \;&=\;Q(x_1x_2-ix_2x_1)
  \;=\;x_3x_2-ix_2x_3=x_{32}\,,\\
  Qx_{32}
  \;&=\;x_{12}\,,\\
  Qx_{1(32)}
  \;&=\;Q(x_1x_{32}+ix_{32}x_1)
  \;=\;x_3x_{12}+ix_{12}x_3
  \;=\;x_{3(12)}\;=\;-x_{1(32)}\,.
\end{align*}
Hence our product formula yields for the graded trace of $Q$:
\begin{align*}
  \trace{Q}{\B(M)}(t)
  \;&=\;\trace{Q}{\langle x_2 \rangle}(t)
  \cdot \trace{Q}{\langle x_1x_3 \rangle}(t)
  \cdot \trace{Q}{\langle x_{12}x_{32} \rangle}(t)
  \cdot \trace{Q}{\langle x_{1(32)}\rangle}(t)\\
  \;&=\;(2)_t\cdot (2)_{-t^2}\cdot (2)_{-t^4}\cdot (2)_{-t^3}
\end{align*}
\end{example}

\subsection{A Non-Normalizing Example with Alternative
 PBW-Basis}\label{sec_Nonnormalizing}

Consider Example \nref{exm:A2} in the previous subsection, which is
not normalized, for $q=-i$. We first calculate the graded trace directly on
the basis $x_1^ix_2^jx_{12}^k$ with $i,j,k\in\{0,1\}$:
\begin{align*}
  \trace{Q}{\B(M)}(t)
  \;&=\;\tr\left(Q|_{1_{\B(M)}}\right)
  \,+\,t\cdot \tr\left(Q|_{\substack{x_1,\\x_2}}\right)
  \,+\,t^2\cdot \tr\left(Q|_{\substack{x_1x_2,\\x_{12}}}\right)\\
  &\qquad \,+\,t^3\cdot \tr\left(Q|_{\substack{x_1x_{12},\\x_2x_{12}}}\right)
  \,+\,t^4\cdot \tr\left(Q|_{x_1x_2x_{12}}\right)\\
  \;&=\;\tr\begin{pmatrix} 1 \end{pmatrix}
  \,+\,t\cdot\tr\begin{pmatrix}  0 & 1 \\ 1 & 0 \end{pmatrix}
  \,+\,t^2\cdot\tr\begin{pmatrix} -i & -2i \\ i  & i \end{pmatrix} \\
  &\qquad \,+\,t^3\cdot\tr\begin{pmatrix}  0 & -i \\ i & 0 \end{pmatrix}
  \,+\,t^4\tr\begin{pmatrix}  1  \end{pmatrix}
  \;=\;1+t^4    
\end{align*}
We observe that in this case we have the following symmetric analogon to a
PBW-basis which explains this graded trace: Denote
$x_+:=x_1+x_2$ and $x_-:=x_1-x_2$.
Then these elements have a common power: 
$$y\;:=\;x_+^2\;=\;-x_-^2\;=\;x_1x_2+x_2x_1 \qquad
z\;:=\;x_+^4\;=\;x_-^4\;=\;2x_1x_2x_1x_2$$
Moreover, we have the relation 
$r:=x_+x_-y=(-x_1x_2+x_2x_1)(x_1x_2+x_2x_1)=-x_1x_2x_1x_2+x_2x_1x_2x_1=0$
and up to $r$ the elements $x_+^ix_-^jy^kz^l$ form a basis. More precisely, we 
have an alternative presentation for the Hilbert series
$$\H(t)=\frac{(2)_t(2)_t(2)_{t^2}(2)_{t^4}}{(2)_{t^4}}=1+2t+2t^2+2t^3+t^4$$
that could be reformulated on the level of graded vector spaces:
$$\B(r\k)\;\rightarrow\; \B(M)\;\rightarrow\;
\B(x_+\k)\,\otimes\, \B(x_-\k)\,\otimes\, \B(y\k)\,\otimes\, \B(z\k)\,.$$
This factorization with relation is stabilized by the action of $Q$ ($x_+,y,z$
even and $x_-,r$ odd), from which we conclude with the formula of Lemma
\nref{lm_ProductFormula_abelian}:
$$\trace{Q}{\B(M)}=\frac{(2)_{t}(2)_{-t}(2)_{t^2}(2)_{t^4}}{(2)_{-t^4}}
=1+t^4\,.$$

\subsection{Factorization Mechanisms In Examples Over Nonabelian
Groups}\label{sec_Nonabelian}

We shall finally consider a family of examples over nonabelian groups obtained
by the first author in \cite{Len13a}: For a finite-dimensional semisimple 
simply-laced Lie algebra $\g$ with a diagram automorphism $\sigma$, consider the
diagonal Nichols algebra $\B(M)$ of type $\g$. We define a covering Nichols
algebra $\B(\tilde{M})$ over a nonabelian group $G$ (an extraspecial $2$-group)
and with folded Dynkin diagram $\g^\sigma$. The covering Nichols algebra is
isomorphic to $\B(M)$ as an algebra; however, there exist nondiagonal Doi
twists, which leave the Hilbert series invariant.

The root system of $\B(\tilde{M})$ is $\g^\sigma$, but because the root spaces
$M_\alpha$ are mostly $2$-dimensional, this cannot explain the full
factorization of the Hilbert series. The old factorization along the
$\g$-root system is not a factorization into sub-Yetter-Drinfel'd modules, but
nevertheless still shows the full factorization.

\begin{example}
  Let $G=\D_4=\langle a,b\;|\;a^4=b^2=1\rangle$ the dihedral group and consider
  the Nichols algebra $\B(\O_b^\chi\oplus \O_{ba}^\psi)$
  where the centralizer characters are
  $\chi(b)=-1$, $\psi(ab)=-1$, and $\chi(a^2)=\psi(a^2)=1$
  (respectively $\chi(a^2)=\psi(a^2)=-1$ for the nondiagonal Doi twist). This
  was the first known example for a finite-dimensional Nichols algebra 
  in \cite{MS00} and it is known to be of type $A_2$. We have
  $\B(\O_b^\chi\oplus \O_{ba}^\psi)\,\cong\, \B(\O_b^\chi)\otimes \B(\O_{ba}^\psi)
  \otimes \B([\O_b^\chi,\O_{ba}^\psi])$
  with respective Hilbert series
  $\H(t)=(2)_t^2\cdot(2)_t^2\cdot(2)_{t^2}^2$.
  From the root system we can only explain the factorization into three
  factors. However, this Nichols algebra is the covering Nichols algebra of a
  diagonal Nichols algebra of type $A_2\cup A_2$ and from this presentation we
  may read off the full factorization in an inhomogeneous PBW-basis:
  \begin{align*}
    \B(\O_b^\chi\oplus \O_{ba}^\psi)
    &\cong \left(\B(x_b+x_{a^2b})\otimes\B(x_{ab}+x_{a^3b})\otimes 
    \B([x_b+x_{a^2b},x_{ab}+x_{a^3b}])\right)\\
    & \otimes\left(\B(x_b-x_{a^2b})\otimes\B(x_{ab}-x_{a^3b})\otimes 
    \B([x_b-x_{a^2b},x_{ab}-x_{a^3b}])\right)\\
    \Rightarrow\quad \H(t)
    &=\big((2)_t(2)_t(2)_{t^2}\big)\cdot \big((2)_t(2)_t(2)_{t^2}\big)
  \end{align*}

\end{example}

  % Section 3: Graded Traces and the Hilbert Series over Nichols algebras
  \subsection{Factorization by Sub-Nichols-Algebras} \label{SEC_factorization_by_subnichols}

During this subsection, let $\k$ be an arbitrary
field, $G$ a finite group and consider a rank-1-\ydm
$M:=\O_g^\chi$ for some $g\in G$ and a one-dimensional representation
$\chi$ of $\Cent(g)$. Denote the conjugacy class of $g$ in $G$
with $X$. Define the {\em enveloping group}
$$ \Env(X) \;:=\; \langle g_x,\,x\in X\,|\,g_xg_y=g_{xyx^{-1}}g_x \rangle\,. $$
The Nichols algebra $\B(M)$ is naturally graded
by $\Env(X)$.
Now $\pi:\Env(X)\twoheadrightarrow \Z_k := \Z/k\Z$,
$g_x\mapsto 1$ for all $x\in X$ establishes $\Z_k$ as a canonical quotient of
$\Env(X)$ for all $k\in\N$.
The original group $G$ is another quotient of $\Env(X)$, and this
induces an action of $\Env(X)$ on $\B(M)$.

Denote the generators of the Nichols algebra by $e_x:=x\ot 1$ and
by $e^*_x$ the dual base.

Define $q_{x,y}$ for $x,y\in X\subset G$ by $c(e_x\ot e_y)
=q_{x,y}\,e_{xyx^{-1}}\ot e_x$ and $m_x\in\N$ minimal such that
$1+q_{x,x}+q_{x,x}^2+\cdots+q_{x,x}^{m_x-1}=0$ for all diagonal elements
$q_{x,x}$. As we only consider rank one, $m_x$ does not
depend on $x$, and we call $m=m_x$ the {\em order} of $q$.
Throughout this section, assume that each coefficient $q_{x,y}$ is a
(not necessarilly primitive) $m$-th root of unity.

In \cite{Lochmann_divisibility}, the second author derived divisibility
relations for the Hilbert series of Nichols algebras by an analysis of
the modified shift
$$ \xi_x:\,\B(M)\rightarrow \B(M), \quad v\;\mapsto
\;(\partial_x^\op)^{m-1}(v)\,+\,e_x\,v $$
and similar maps, where $x\in X$ is arbitrary
and $\partial_x^\op=(e_x^*\ot\id)\Delta$ is the opposite braided derivation.

For each $x\in X$, $\xi_x$ is a linear isomorphism,
leaving $\ker\partial_y$ invariant for all $y\in X\setminus\{x\}$, where
$\partial_y=(\id\ot e_y^*)\Delta$ is the braided derivation of
Definition \ref{DEF_skew_derivations}. If $\Xi$ is the
group generated by all $\xi_x$, $x\in X$, the orbit of $1\in\B(M)$ under
$\Xi$ linearly spans $\B(M)$. Finally, let $\pi:\Env X\twoheadrightarrow H$ be
some group epimorphism, such that $\pi(g_x)^m=e$. Then $\xi_x$ maps
the $\B(M)$-layer of degree $h$ to the layer of degree $\pi(g_x)h$ for all $h\in H$,
see \cite{Lochmann_divisibility}, Proposition 9. For us, the relevant quotient
$H$ of $\Env X$ will not be $G$, but $\Z_m$ as chosen above.

In addition, for each $x,y\in X$, $\xi_x$ satisfies $g_y\circ\xi_x=
q_{y,x}\cdot\xi_{y\trid x}\circ g_y$, where we identify
$g_y$ with the action of $g_y$ on $\B(M)$:
\begin{align*}
g_y.(e_x\cdot v)
\;&=\; q_{y,x}\,e_{y\trid x}\cdot (g_y.v) \\
\textnormal{and}\quad g_y.(\partial_x^\op(v))
\;&=\; (e_x^*\ot g_y)\Delta(v)
\;=\; q_{y,x}^{-1}(e_{y\trid x}^*\ot\id) (g_y\ot g_y)\Delta(v) \\
\;&=\; q_{y,x}^{-1}\cdot \partial_{y\trid x}^\op(g_y.v)\,, \\
\textnormal{thus}\quad g_y\circ(\partial_x^\op)^{m-1}
\;&=\; q_{y,x}^{(-1)\cdot(m-1)}\cdot
(\partial_{y\trid x}^\op)^{m-1}\circ g_y
\;=\; q_{y,x}\cdot
(\partial_{y\trid x}^\op)^{m-1}\circ g_y
\end{align*}
for all $v\in \B(M)$ (the second equality is due to $e_{y\trid x}^*(g_y.e_z)
=q_{y,z}\delta_{y\trid x, y\trid z}=q_{y,z}\delta_{x,z}=
q_{y,x}\cdot\delta_{x,z}=q_{y,x}\cdot e_x^*(e_z)$).

\begin{lemma} \label{LEM_divisibility_from_balancedness}
Let $\lambda$ be a $k$-th root of unity (not necessarilly primitive).
Let $M$ be some finite-dimensional graded vector space and $Q\in\Aut M$.
Consider $\Z_k=\Z/k\Z$ as a quotient of $\Z$, then $M$ has a $\Z_k$-grading.
With respect to this grading,
$\tr Q|_{M([j+1]_k)}=\lambda\tr Q|_{M([j]_k)}$ for all $j\in \Z$ if and only if the $\Z$-graded
character $\trace{Q}{M}(t)$ is divisible by $(k)_{\lambda t}$.
\end{lemma}
\begin{proof}
This is a straight-forward generalization of Lemma 6 of \cite{Lochmann_divisibility}:
Given a polynomial $p$, denote with $p_j$ the coefficient of $t^j$ in $p(t)$
(or zero if $j<0$). Set $b_j:=\tr Q|_{M(j)}$.

``$\Rightarrow$'': Choose $p_j\,:=\,0$ for each $j\in\Z_{<0}$ and 
$p_j:=b_j-\sum_{i=1}^{k-1}\,\lambda^ip_{j-i}$, hence $b_j-\lambda b_{j-1} =
p_j-\lambda^kp_{j-k} = p_j-p_{j-k}$.
Let $d\,\in\,\N_0$ be such that $d\cdot k$ is larger than the top degree of $M$.
Summation of the previous equation then yields for each $0\,\leq\,l\,\leq k\,-\,1$
the telescoping sum
\begin{eqnarray*}
\sum_{0\,\leq\,j\,\leq\,d}b_{jk+l}\,-\,\lambda\cdot\sum_{0\,\leq\,j\,\leq\,d}b_{jk+l-1}
&=& -\,p_{l-k}\,+\,p_{dk+l}\,.
\end{eqnarray*}
The two sums on the left hand side sum to $\tr Q|_{M([l]_k)}$ and
$\tr Q|_{M([l-1]_k)}$,
respectively, so by assumption, the left hand side is zero. $p_{l-k}$ is zero by
definition ($l-k<0$), hence $p_{dk+l}$ is zero. This proves that $p$ is a polynomial,
and from $b_j\,=\,\sum_{i=0}^{k-1}\,\lambda^i p_{j-i}$ follows
$\trace{Q}{M}(t)\,=\,(k)_{\lambda t} \cdot p(t)$.

``$\Leftarrow$'': Let $\trace{Q}{M}(t)=(k)_{\lambda t}\cdot p(t)$ for some polynomial $p$.
We have $b_j=\sum_{i=0}^{k-1}\,\lambda^i p_{j-i}$ and therefore for each $[l]_k\,\in\,\Z_k$
\begin{eqnarray*}
\tr Q|_{M([l]_k)} &=& \!\!\!\!\!\!\!\!\!\!\!\!
\sum_{\begin{array}{c} j\in\N_0,\\ j \equiv l\; (\mathrm{mod}\, k)\end{array}}
\!\!\!\!\!\!\!\!\!\!\!\! \tr Q|_{M(j)} \;\;=\;\; \!\!\!\!\!\!\!\!\!\!\!\!
\sum_{\begin{array}{c} j\in\N_0,\\ j \equiv l\; (\mathrm{mod}\, k)\end{array}}
\!\!\!\!\!\!\!\! \sum_{i=0}^{k-1}\,\lambda^i p_{j-i}
\;\;=\;\;\sum_{j\in\N_0}\,\lambda^{l-j}p_j\,,
\end{eqnarray*}
from which follows $\tr Q|_{M([l+1]_k)}=\lambda\cdot\tr Q|_{M([l]_k)}$.
\end{proof}

\begin{theorem} \label{THE_divisibility_by_subnichols}
Let $G$ be a finite group, $G'\subset G$ a proper subgroup, and $h\in G'$
be arbitrary. Let $\chi$ be a one-dimensional representation
of $\Cent_G(h)$, and let $\chi'$ be its restriction to
$\Cent_{G'}(h)=\Cent_G(h)\cap G'$. Set $M:=\O_h^\chi$ and
$M':=\O_h^{\chi'}$. Set $X$ and $X'$ to be the conjugacy classes
of $h$ in $G$ and $G'$, respectively. Let $g\in\Env(X')$ be arbitrary
and identify $g$ with its actions on $\B(M)$ and $\B(M')$.
%Let $\B(M)$ be a finite-dimensional Nichols algebra over a quandle $X$
%and a 2-cocycle $q$ of order $m$, such that each $q_{ij}^m=1$ for all
%$i,j\in X$. Let $X'$ be a non-empty proper subrack
%of $X$ and $\B(M')$ its corresponding Nichols sub-algebra of $\B(M)$.
%Let $g\in\Env(X')$ be arbitrary.

1) Then $\trace{g}{\B(M)}(t)$ is divisible by $\trace{g}{\B(M')}(t)$.

2) Assume there is some $x\in X$, such that $g\circ\xi_x=\lambda\cdot\xi_x\circ g$
for some $m$-th root of unity $\lambda$, where $m$ is the order of $q$.
Then $\trace{g}{\B(M)}(t)$ is divisible by
$(m)_{\lambda t}\cdot \trace{g}{\B(M')}(t)$.
\end{theorem}
\begin{proof}
Set $K:=\bigcap_{x\in X'}\,\ker\partial_x$.

1) $\B(M)$ is free as a $\B(M')$-module, so there is a linear isomorphism
$\B(M)\cong K\ot\B(M')$ mediated by multiplication
(e.g.~\cite{Skr_freeness}, \cite{G_freeness}).
$K$ and $\B(M')$ are both closed under the action of $\Env(X')$ ($K$
is closed because $X'$ is closed under conjugation). Therefore,
$\B(M)\cong K\ot\B(M')$ as $\Env(X')$-representations and
$\trace{g}{\B(M)}(t)=\trace{g}{K}(t)\cdot \trace{g}{\B(M')}(t)$.

2) We show that $\trace{g}{K}(t)$ is divisible by $(m)_{\lambda t}$.
Set $K_j:=K\cap \B(M)_j$ (layer $j$ of $\B(M)$ with $j\in\Z_m$).
The modified shift operator $\xi_x$ establishes a linear
isomorphism between $K_j$ and $K_{j+1}$ for each $j\in\Z_m$.
Let $B$ be a basis for $K_j$ and $B':=\xi_x(B)$, and denote with
$v^*$ the dual basis element corresponding to $v\in B$ for the basis $B$
and $v\in B'$ for the basis $B'$, respectively. Then
$$
\tr g|_{K_{j+1}} \;=\; \sum_{v\in B'} v^*(g.v)
\;=\; \sum_{b\in B} b^*(\xi_x^{-1} g \xi_x(b))
\;=\; \lambda\sum_{b\in B} b^*(g.b)
\;=\; \lambda\tr g|_{K_j}
$$
holds. Apply Lemma \ref{LEM_divisibility_from_balancedness}.
\end{proof}

The condition $g\circ\xi_x=\lambda\cdot\xi_x\circ g$ of part (2) of Theorem
\ref{THE_divisibility_by_subnichols} is fulfilled for $gx=xg$ and 
$\lambda=q_{y_1,x}\cdots q_{y_s,x}$ with
$g=g_{y_1}\cdots g_{y_s}$, $y_1,\ldots,y_s\in X$.

\begin{example}
Choose $G'=\Perm{3}\subset G=\Perm{4}$ and $h\in G'$ a transposition,
so $X$ and $X'$ are the conjugacy classes of transpositions.
Choose $\chi$ and $\chi'$ to be the alternating representations of
$G$ and $G'$. Their Nichols algebras will appear again in
Subsections \ref{SEC_6ABC} and \ref{SEC_3A}, respectively. 
Choose $g=(1\,2)$ and $x=(3\,4)$. Then $g$ and $x$ commute, and Theorem
\ref{THE_divisibility_by_subnichols} explains why
$\trace{g}{\B(M)}(t)=(2)_{-t}^4(3)_t^2(2)_{t^4}$ contains the factor
$(2)_{-t}\cdot\trace{g}{\B(M')}(t)=(2)_{-t}^3(3)_t$.
\end{example}

\section{Calculations for Small Rank-1 Nichols Algebras}

The following results have been calculated with the help of
GAP (\cite{GAP}) in a straight-forward way: First, calculate a linear
basis for the given Nichols algebra, then generate the representing
matrix of the action of each element of the conjugacy class $X$,
which also generates $G$, and then calculate the graded traces of
all conjugacy classes.

The Nichols algebra of Subsection \ref{SEC_4C} admits a
large dimension of 5,184. For this size, it was not possible for us
to calculate all matrices we needed. In this special case, we made
use of Corollary \ref{COR_poincare_duality}, so we could restrict
our matrix calculations to the lower half of grades and compute the
full graded trace by Poincar\'e duality.

The Nichols algebras of dimensions 326,592 and 8,294,400 are
computationally not yet accessible with this method.

\subsection{\texorpdfstring
{$\dim M=3$, $\dim\B(M)=12$}
{dim(M)=3, dim(B(M))=12}
} \label{SEC_3A}

Let $G=\Perm{3}$ and $g=(1\,2)$, then the centralizer of $g$ is
isomorphic to $\Z/2\Z$, and we choose $\chi$ to be its alternating
representation. Then $\B(\O_g^\chi)$ is a 12-dimensional Nichols
algebra on which $G$ acts faithfully with the following
graded characters:
\begin{align*}
\traceBVt{e}      \;&=\; (2)_t^2(3)_t &
\traceBVt{(1\,2)}    \;&=\; (2)_{-t}^2(3)_t &
\traceBVt{(1\,2\,3)} \;&=\; (2)_{-t}^2(2)_t^2
\end{align*}
From this, we can calculate the decomposition into irreducible
$G$-represen\-tations in each degree. If we denote the trivial irreducible,
alternating, and standard $G$-representations with $T$, $A$, and
$S$, respectively, we find
\begin{eqnarray*}
\B(\O_g^\chi) &\cong& T \;\oplus\; (A\oplus S)t \;\oplus\;
2St^2 \;\oplus\; (A\oplus S)t^3 \;\oplus\; Tt^4 \\
&\cong& (T \;\oplus\; At) \;\otimes\; (T \;\oplus\; St
\;\oplus\; St^2 \;\oplus\; At^3)
\end{eqnarray*}
as $G$-representation. The factorization in line 2 results from
a certain sub-Nichols-algebra (see \ref{SEC_factorization_by_subnichols})
and implies the factorizations
\begin{eqnarray*}
\traceBVt{e}      &=& (2)_{t\phantom{-}}\;\cdot\;(2)_t(3)_t \\
\traceBVt{(1\,2)}    &=& (2)_{-t}\;\cdot\;(2)_{-t}(3)_t \\
\traceBVt{(1\,2\,3)} &=& (2)_{t\phantom{-}}\;\cdot\;(2)_{-t}^2(2)_t
\end{eqnarray*}
of the graded characters. To understand the factorization of
the remaining terms $(2)_t(3)_t$, $(2)_{-t}(3)_t$, and $(2)_{-t}^2(2)_t$,
this line of argument, however, fails, because $T \;\oplus\; St
\;\oplus\; St^2 \;\oplus\; At^3$ does not factor into a tensor product
of $G$-representations (we have already met this phenomenon in another
situation in Section \ref{SEC_toy_example}). For $g=e$ and
$g=(1\,2)$, we may apply Theorem \ref{THE_divisibility_by_subnichols}.(2)
to explain the additional factors $(2)_t$ and $(2)_{-t}$, respectively, but
this neither helps in the case $g=(1\,2\,3)$, nor to understand the origin
of the factors $(3)_t$ for $g\in\{e,\,(1\,2)\}$.

If $\B(\O_g^\chi)$ does not factor as a $G$-representation, one might think that
it may still factor as an $\langle h\rangle_G$-representation for
each $h\in G$, which would explain the factorization of the graded characters
just as well. This, however, is wrong: Take $h=(1\,2\,3)$, which is of order $3$.
Let $T$ be the trivial irreducible representation, $B$ one of the non-trivial
irreducible representations, and set $C:=B\otimes B$. Then $\B(\O_g^\chi)$ is
\begin{eqnarray*}
\B(\O_g^\chi) &\cong& T \;\oplus\; (T\oplus B\oplus C)t \;\oplus\;
(2B\oplus 2C)t^2 \;\oplus\; (T\oplus B\oplus C)t^3 \;\oplus\; Tt^4 \\
&\cong& (T \;\oplus\; Tt) \;\otimes\; (T \;\oplus\; (B\oplus C)t
\;\oplus\; (B\oplus C)t^2 \;\oplus\; Tt^3)
\end{eqnarray*}
as an $\langle h\rangle_G$-representation and
$T\oplus(B\oplus C)t\oplus(B\oplus C)t^2\oplus Tt^3$
does not factor further.

%\begin{remark}
%  Decompose $M=T\oplus St$ as $G$-module, then by Lemma \nref{lm_braidedSub} it
%  even decomposes as a braided vector space. Hence we observe, that
%  $$\Delta^+:=\{\alpha_T,\alpha_{St}\}\qquad M_{\alpha_T}=T,M_{\alpha_{St}}=St$$
%  is a PBW-Basis of braided vector spaces, i.e. we observe that multiplication
%  in $\B(M)$ yields the following isomorphism of graded vector spaces
%  $$\B(M_T)\B(M_{St})\stackrel{\sim}{\rightarrow}\B(M)$$
%  This PBW-Basis proves as already pointed out above the factorization of the
%  graded trace as follows
%  $$\trace{g}{\B(M)}=\trace{g}{\B(M_T)}\trace{g}{\B(M_{St})}
%  =[2]_t\cdot\left([2]_t[3]_4\right)$$
%  However, the further decomposition of the trace on the simple braieded
%  vector space $M_{St}$ cannot be explained by this.\\
%
%  Let us hence consider the action of $\B(M_T)$ on the coinvariants
%  $\B(M)^{coinv}$ of the restriction map $\pi_\B(M)\rightarrow
%  \B(M_T)$......
%\end{remark}

%$G$ admits the epimorphism $\mathrm{sgn}:G\twoheadrightarrow \Z_2$ which
%induces the following graded characters of the adjoint representation
%of $G$ on $\K G$:
%\begin{eqnarray*}
%\chi'(e)      &=& 3(t+1) \\
%\chi'(g_1)    &=& t+1 \\
%\chi'(g_1g_2) &=& 3
%\end{eqnarray*}

\subsection{\texorpdfstring
{$\dim M=3$, $\dim\B(M)=432$}
{dim(M)=3, dim(B(M))=432}
} \label{SEC_3B}

Assume $\cha \K=2$ and $\K$ admits a primitive third root of unity $\zeta$.
Choose
$$ G \;=\; \langle g_1,\, g_2 \;:\; g_1^6,\, g_2^6,\, (g_1g_2)^3,\,
g_1^2g_2^{-2}
\rangle \;\cong\; \Z_3\times \Perm{3} $$
and $g=g_1$. The centralizer of $g$ is
$\langle g_1 \rangle \,\cong\,\Z_3\times\Z_2$. Choose $\chi(g_1)=\zeta$.
Then $\B(\O_g^\chi)$ is a faithful $G$-representation of dimension 432.
$G$ has nine conjugacy classes, and we choose
$\{e, g_1, g_1^2, g_1^3, g_1^4, g_1^5, g_1g_2, g_1^3g_2, g_1^5g_2\}$
as their representatives.
Then the graded characters of $\B(\O_g^\chi)$ (not Brauer characters, but with
values in $\K$) are:
\begin{align*}
\traceBVt{\,e}        \,&=\, (2)_t^6(3)_t^7 &
\traceBVt{\,g_1}      \,&=\, (2)_t^7(3)_t^6(2)_{\zeta^2 t} \\
\traceBVt{\,g_1^2}    \,&=\, (2)_t^7(3)_t^6(2)_{\zeta t} &
\traceBVt{\,g_1^3}    \,&=\, (2)_t^6(3)_t^7 \\
\traceBVt{\,g_1^4}    \,&=\, (2)_t^7(3)_t^6(2)_{\zeta^2 t} &
\traceBVt{\,g_1^5}    \,&=\, (2)_t^7(3)_t^6(2)_{\zeta t} \\
\traceBVt{\,g_1g_2}   \,&=\, (2)_t^{10}(2)_{\zeta t}^{10} &
\traceBVt{\,g_1^3g_2} \,&=\, (2)_t^{10}(2)_{\zeta^2 t}^{10} \\
\traceBVt{\,g_1^5g_2} \,&=\, (3)_t^{10}
\end{align*}

\subsection{\texorpdfstring
{$\dim M=4$, $\dim\B(M)=36$ or $72$}
{dim(M)=4, dim(B(M))=36 or 72}
} \label{SEC_4AB}

%Let $X$ be the quandle given by:
%\begin{eqnarray*}
%\begin{array}{|c|cccc|} \hline
%x & g_x(1) & g_x(2) & g_x(3) & g_x(4) \\ \hline
%1 & 1 & 3 & 4 & 2 \\
%2 & 4 & 2 & 1 & 3 \\
%3 & 2 & 4 & 3 & 1 \\
%4 & 3 & 1 & 2 & 4 \\ \hline
%\end{array}
%\end{eqnarray*}
Consider $\cha \k=2$, $G=\APerm{4}$ and $g=(1\,2\,3)$.
The centralizer of $g$ is isomorphic to $\Z_3$, choose
$\chi$ to be the trivial irreducible representation.
Then $\B(\O_g^\chi)$ is 36-dimensional with graded characters
(not Brauer characters):
\begin{align*}
\traceBVt{\,e}            \;&=\; (2)_t^2(3)_t^2 &
\traceBVt{\,(1\,2\,3)}    \;&=\; (2)_t^4(3)_t \\
\traceBVt{\,(1\,2)(3\,4)} \;&=\; (2)_t^2(3)_t^2 &
\traceBVt{\,(1\,3\,2)}    \;&=\; (2)_t^4(3)_t 
\end{align*}
In characteristic $\neq 2$, there is a very similar Nichols algebra
of dimension 72: Assume $\cha \k\neq 2$ and
$$ G \;=\; \langle g_1,\, g_2 \;:\; g_1^6,\, g_2^6,\, [g_1^3,g_2],\,
(g_1g_2)^3,\, (g_1g_2^2)^2 \rangle \;\cong\; \APerm{4}\times\Z_2\,. $$
Choose $g=g_1$, then the centralizer is
$\langle g_1\rangle\cong \Z_3\times\Z_2$, to which we choose
the representation $\chi(g_1):=-1$. Then the graded characters
of $\B(\O_g^\chi)$ are:
\begin{align*}
\traceBVt{\,e}          \;&=\;  (2)_t^3   (3)_{-t}  (3)_t^2 &
\traceBVt{\,g_1}        \;&=\;  (2)_{-t}^3(2)_t^2   (3)_{-t}(3)_t \\
\traceBVt{\,g_1^2}      \;&=\;  (2)_{-t}^2(2)_t^3   (3)_{-t}(3)_t &
\traceBVt{\,g_1^3}      \;&=\;  (2)_{-t}^3(3)_{-t}^2(3)_t \\
\traceBVt{\,g_1^4}      \;&=\;  (2)_{-t}^2(2)_t^3   (3)_{-t}(3)_t &
\traceBVt{\,g_1^5}      \;&=\;  (2)_{-t}^3(2)_t^2   (3)_{-t}(3)_t \\
\traceBVt{\,g_1^4g_2^2} \;&=\;  (2)_t^3   (3)_{-t}^3 &
\traceBVt{\,g_1^2g_2}   \;&=\;  (2)_{-t}^3(3)_t^3
\end{align*}
Consider the subgroup $H:=\langle g_1^2, g_2^2\rangle_G \cong \APerm{4}$
of $G$. The graded characters of the $H$-action on $\B(\O_g^\chi)$ are exactly
those of the left column in the above list. 
If considered in characteristic 2, these polynomials are divisible by
the corresponding graded characters of the 36-dimensional Nichols
algebra, with $(2)_t(3)_t$ as common quotient.

\subsection{\texorpdfstring
{$\dim M=4$, $\dim\B(M)=5184$}
{dim(M)=4, dim(B(M))=5184}
} \label{SEC_4C}

Assume $\cha\k\neq 2$ and that $\k$ admits a primitive third root of unity $\zeta$.
Choose
$$ G \;:=\; \langle a,\,b\;:\;a^3\,=\,b^3\,=\,(ab)^2\rangle \;\cong\; \SL(2,3)\,
$$
and $g=a^4$. The centralizer of $g$ is $\langle a\rangle \cong \Z_6$.
Choose the representation $\chi(a):=-\zeta$.
%define $a:=(g_1g_3)^3g_1$ and $b:=(g_1g_3)^3g_3^2$ to find the presentation
%where $a^3$ is $(-1)^l\cdot\id$ in degree $l$.
%With this choice, $g_1=a^3b$, $g_2=a^2ba$, $g_3=b^{-1}aba^2b$, and $g_4=aba^2$.
This leads
to the following graded characters of $\B(\O_g^\chi)$:
\begin{eqnarray*}
\traceBVt{e}   &=& (2)_t^4    (2)_{t^2}^2 (3)_{t^2}^4 \\
\traceBVt{a}   &=& (2)_t^4    (2)_{-t}^4  (2)_{t^2}^2  (2)_{\zeta t}  
(2)_{-\zeta t}^2 (2)_{\zeta^2 t}^3 (2)_{-\zeta^2 t}^3 (2)_{-\zeta t^3} \\
\traceBVt{a^2} &=& (2)_t^4    (2)_{-t}^4  (2)_{t^2}^2  (2)_{\zeta t}^3
(2)_{-\zeta t}^3 (2)_{\zeta^2 t}^2 (2)_{-\zeta^2 t}   (2)_{\zeta^2 t^3} \\
\traceBVt{a^3} &=& (2)_{-t}^4 (2)_{t^2}^2 (3)_{t^2}^4 \\
\traceBVt{a^4} &=& (2)_t^4    (2)_{-t}^4  (2)_{t^2}^2  (2)_{\zeta t}^2
(2)_{-\zeta t}   (2)_{\zeta^2 t}^3 (2)_{-\zeta^2 t}^3 (2)_{\zeta t^3} \\
\traceBVt{a^5} &=& (2)_t^4    (2)_{-t}^4  (2)_{t^2}^2  (2)_{\zeta t}^3
(2)_{-\zeta t}^3 (2)_{\zeta^2 t}   (2)_{-\zeta^2 t}^2 (2)_{-\zeta^2 t^3} \\
\traceBVt{ab}  &=& (2)_{t}^4  (2)_{-t}^4  (3)_{t^2}^4 
\end{eqnarray*}
In characteristic 2, $G=\APerm{4}$ yields a Nichols algebra with the same
Hilbert series. Some of the above conjugacy classes merge in this case, because
$\SL(2,3)$ is a $\Z_2$-extension of $\APerm{4}$, but apart from that,
the resulting graded characters are the same as above.

\subsection{\texorpdfstring
{$\dim M=5$, $\dim\B(M)=1280$}
{dim(M)=5, dim(B(M))=1280}
} \label{SEC_5AB}

Choose
$$ G \;:=\; \langle a,\;b \;:\;
a^4,\;b^4,\;ab^3a^2b^2\rangle $$
and $g:=a$. $G$ is isomorphic to the GAP's small group number 3 of size 20
(\cite{GAP}), a semi-direct product of $\Z_5$ and $\Z_4$.
The centralizer of $g$ is $\langle a\rangle\cong \Z_4$.
Choose the representation $\chi(a):=-1$.
Then $\B(\O_g^\chi)$ is a faithful $G$-representation of dimension $1280$
with the following graded characers:
\begin{align*}
\traceBVt{e}    \;&=\; (2)_t^4            (2)_{t^2}^4           (5)_t &
\traceBVt{a}    \;&=\; (2)_{-t}^4 (2)_t^2 (2)_{t^2}   (2)_{t^4} (5)_t \\
\traceBVt{a^2}  \;&=\; (2)_{-t}^4 (2)_t^4 (2)_{t^2}^2           (5)_t &
\traceBVt{a^3}  \;&=\; (2)_{-t}^4 (2)_t^2 (2)_{t^2}   (2)_{t^4} (5)_t \\
\traceBVt{a^3b} \;&=\; (2)_{-t}^4 (2)_t^4 (2)_{t^2}^4
\end{align*}
There appears a second, non-isomorphic (but dual) Nichols algebra if one chooses
$g:=a^3$, $\chi(a):=-1$ instead (see Example 2.1 in \cite{AFGV_simple}).
It features the same graded characters as $\B(\O_a^\chi)$ above.

\subsection{\texorpdfstring
{$\dim M=6$, $\dim\B(M)=576$} 
{dim(M)=6, dim(B(M))=576}
} \label{SEC_6ABC}

There are three pairwise non-isomorphic cases to consider with $\dim M=6$
and $\dim\B(M)=576$.

First, choose $G=\Perm{4}$ and $g:=(1\,2)$. The centralizer of $g$
is $\langle(1\,2),\,(3\,4)\rangle\cong \Z_2/\Z_2$. Choose
the representation with $\chi((1\,2))=-1$ and $\chi((3\,4))=-1$.
The graded characters of $\B(\O_g^\chi)$ are:
\begin{align*}
\traceBVt{e}             \;&=\; (2)_t^4            (2)_{t^2}^2 (3)_t^2 &
\traceBVt{(1\,2)}        \;&=\;         (2)_{-t}^4 (2)_{t^4}   (3)_t^2 \\
\traceBVt{(1\,2)(3\,4)}  \;&=\; (2)_t^4 (2)_{-t}^4             (3)_t^2 &
\traceBVt{(1\,2\,3)}     \;&=\; (2)_t^4 (2)_{-t}^4 (2)_{t^2}^2         \\
\traceBVt{(1\,2\,3\,4)}  \;&=\; (2)_t^2 (2)_{-t}^4 (2)_{t^2}   (3)_t^2
\end{align*}
Now choose the representation $\chi((1\,2))=-1$, $\chi((3\,4))=1$
instead. Then the graded characters of $\B(\O_g^\chi)$ are:
\begin{align*}
\traceBVt{e}             \;&=\; (2)_t^4            (2)_{t^2}^2 (3)_t^2  &
\traceBVt{(1\,2)}        \;&=\; (2)_t^2 (2)_{-t}^4 (2)_{t^2}   (3)_t^2  \\
\traceBVt{(1\,2)(3\,4)}  \;&=\;         (2)_{-t}^4 (2)_{t^2}^2 (3)_t^2  &
\traceBVt{(1\,2\,3)}     \;&=\; (2)_t^4 (2)_{-t}^4 (2)_{t^2}^2          \\
\traceBVt{(1\,2\,3\,4)}  \;&=\; (2)_t^2 (2)_{-t}^4 (2)_{t^2}   (3)_t^2
\end{align*}
%First, choose 
%$$ G \;:=\; \langle g_1,\, g_2,\, g_3 \;:\; g_1^2,\, g_2^2,\, g_3^2,\,
%(g_1g_2)^2,\, (g_1g_3)^3,\, (g_2g_3)^3\rangle \;\cong\; \Perm{4} $$
Third, choose $G=\Perm{4}$ and $g:=(1\,2\,3\,4)$. The centralizer
of $g$ is $\langle(1\,2\,3\,4)\rangle\cong\Z_4$. Choose
the representation $\chi((1\,2\,3\,4))=-1$. Then $\B(\O_g^\chi)$
has the following graded characters:
\begin{align*}
\traceBVt{e}             \;&=\; (2)_t^4            (2)_{t^2}^2 (3)_t^2 &
\traceBVt{(1\,2)}        \;&=\; (2)_t^2 (2)_{-t}^4 (2)_{t^2}   (3)_t^2 \\
\traceBVt{(1\,2)(3\,4)}  \;&=\; (2)_t^4 (2)_{-t}^4             (3)_t^2 &
\traceBVt{(1\,2\,3)}     \;&=\; (2)_t^4 (2)_{-t}^4 (2)_{t^2}^2         \\
\traceBVt{(1\,2\,3\,4)}  \;&=\;         (2)_{-t}^4 (2)_{t^4}   (3)_t^2
\end{align*}
Note how the graded characters differ pairwise for these three cases,
a simple way to see that the three Nichols algebras obtained are
non-isomorphic, not even as $\Perm{4}$-representations, although the
first and the second case are twist-equivalent to each other
(\cite{V_twist}).

%\subsection{$\dim M=3$, $\dim \B(M)=16$} \label{SEC_NonCrossed}
%
%Let $X$ be the quandle defined by:
%\begin{eqnarray*}
%\begin{array}{|c|ccc|} \hline
%x & g_x(1) & g_x(2) & g_x(3) \\ \hline
%1 & 1 & 3 & 2  \\
%2 & 1 & 2 & 3  \\
%3 & 1 & 2 & 3  \\ \hline
%\end{array}
%\end{eqnarray*}
%$X$ is not a crossed set, and therefore does not come from
%a conjugacy class of a group. However, it still defines a braiding
%on $M=\K X$, and if one chooses the constant cocycle $q=-1$, a Nichols
%algebra $\B(M)$ of dimension 16 emerges. It is generated by the relations
%\begin{eqnarray*}
%0 &=& e_1^2 \;\;=\;\; e_2^2 \;\;=\;\; e_3^2 \\
%&=& e_1e_2 + e_1e_3 + e_2e_1 + e_3e_1 \;\;=\;\; e_2e_3 + e_3e_2 \\
%&=& e_1e_2e_1e_2 + e_2e_1e_2e_1\,.
%\end{eqnarray*}
%$X$ is not a crossed set, and because of this, $g_2$ and $g_3$ are equal
%in the exterior group. The acting group then is $G\cong \Z_2\times \Z_2$,
%generated by $g_1$ and $g_2=g_3$ (which are not conjugated, because
%$X$ is decomposable), with the following graded characters:
%\begin{eqnarray*}
%\traceBVt{e}      &=& (1+t)^3        (1+t^2)   \\
%\traceBVt{g_1}    &=& (1+t)^2 (1-t)^3          \\
%\traceBVt{g_2}    &=&         (1-t)^3 (1+t^2)  \\
%\traceBVt{g_1g_2} &=& (1+t)^3 (1-t)^2
%\end{eqnarray*}

\subsection{Observations}

From the examples of the previous sections, we derive the following
observations,
which may help us in understanding the factorization of the Hilbert series and
graded characters of any Nichols algebra. A theory of the representations
coming from a Nichols algebra should be able to explain all of them.

\begin{enumerate}
\item
The zeros of the graded characters of all examples above are
$n$-th roots of unity, where $n$ most of the time is a
divisor of $\# G$, but not always: In Subsection \ref{SEC_4C}, ninth roots of
unity appear though $\# G=24$; and in Subsection \ref{SEC_5AB}, we have $\#
G=20$,
but $\traceBVt{a}$ has an eighth root of unity.
A deeper understanding why there are only roots of unity and which roots
appear how often is eligible.

\item
Each of the characters $\traceBVt{g}$ with $g\neq e$ includes a factor $1-t$ (or
$1+t$ in characteristic $2$), therefore the non-graded character
$\traceBVt{g}(1)$ vanishes.
From this follows that all of the above Nichols algebras are (seen as
their respective $G$-representations) multiples of the regular
$G$-representation.
The only exceptions to this are the 432-dimensional and the 72-dimensional
Nichols algebras of Subsections \ref{SEC_3B} and \ref{SEC_4AB}, each of which
admits a single non-trivial conjugacy class with non-vanishing character.

\item
The smallest common multiple $p$ of the graded characters of a single
Nichols algebra has a surprisingly small degree. We want to point out
that the quotient $p/\traceBVt{g}$ typically is a polynomial whose roots
have the same order as $g$ has in $G$.

\item
Although all of the characters factor nicely (see point (1)), there is no
corresponding factorization of the respective representations; we showed
this in Subsection \ref{SEC_3A}.
\end{enumerate}

     % Section 4: Calculations for Small Rank-1 Nichols Algebras

\section*{Acknowledgement}
The authors wish to thank I.~Heckenberger 
for stimulating discussions.

\bibliographystyle{abbrv}
\bibliography{GradedTraces}

\end{document}